\documentclass[reqno,12pt]{amsart}
\usepackage[utf8x]{inputenc}
\usepackage[unicode]{hyperref}
\usepackage{mathrsfs,bbm}
\usepackage{graphicx}

\usepackage[english]{babel}
\usepackage{verbatim}
\usepackage{supertabular}
\usepackage{stmaryrd}
\usepackage{longtable}
\usepackage{varioref}
\usepackage{calc}
\usepackage{ifthen}
\usepackage{indentfirst}
\usepackage{fancyhdr}
\usepackage{enumerate}
\usepackage{float}
\usepackage{xcolor}
\usepackage{multicol}
\usepackage[BCOR4mm,DIV10]{typearea}
\usepackage[refpage]{nomencl/nomencl}
\usepackage{amssymb}
\usepackage{amsxtra}
\usepackage{amsmath} 
\usepackage{amsfonts}
\usepackage{etoolbox}
\usepackage{xfrac}
\numberwithin{equation}{section}
\allowdisplaybreaks

\newtheorem{theorem}{Theorem}[section]
\newtheorem{proposition}[theorem]{Proposition}
\newtheorem{lemma}[theorem]{Lemma}
\newtheorem{corollary}[theorem]{Corollary}
\newtheorem{definition}[theorem]{Definition}

\newtheorem{example}[theorem]{Example}

\newtheorem*{theorem*}{Theorem}
\hypersetup{
breaklinks=true,
colorlinks=true,
linkcolor=blue,
citecolor=blue,
urlcolor=blue,
}

\newcommand{\abs}[1]{\ensuremath{\left\lvert#1\right\rvert}}
\newcommand{\br}[1]{\ensuremath{\left(#1\right)}}
\newcommand{\seqn}[2][k]{\ensuremath{\br{{#2}_{#1}}_{#1\in\N}}}
\newcommand{\set}[1]{\ensuremath{\left\{#1\right\}}}
\newcommand{\sett}[2]{\ensuremath{\left\{#1\,\middle|\,#2\right\}}}

\newcommand{\norm}[2][]{\left \lVert #2 \right \rVert_{#1}}

\DeclareMathOperator{\TC}{TC}

\DeclareMathOperator{\BiLip}{BiLip}

\DeclareMathOperator{\spann}{span}

\newcommand{\TL}{\mathcal{T}}

\newcommand\Id{{{\rm Id}}}

 \renewcommand\S{{\mathbb S}}

\newcommand\CK{{\mathscr{C}(\mathcal{K})}}

\newcommand{\Fo}{\,\,\,\text{for }\,\,}
\newcommand{\Foa}{\,\,\,\text{for all }\,\,}
\newcommand{\AND}{\,\,\,\text{and }\,\,}

\newcommand{\eps}{\ensuremath{\varepsilon}}

\newcommand{\F}[1][\eps]{F}%

\newcommand{\g}{\ensuremath{\gamma}}

\newcommand{\length}{\mathscr{L}}

\newcommand{\N}{\ensuremath{\mathbb{N}}}

\newcommand{\R}{\ensuremath{\mathbb{R}}}
\renewcommand{\rho}{\ensuremath{\varrho}}

\DeclareMathOperator{\sign}{sign}

\newcommand{\TP}{\mathrm{TP}}
\newcommand{\tpc}[1][\varphi]{\mathrm{tpc}_{#1}}

\newcommand{\vth}{\ensuremath{\vartheta}}

\newcommand{\Z}{\ensuremath{\mathbb{Z}}}

\graphicspath{{./img/}}

\title{Symmetric elastic knots}
\author{Alexandra Gilsbach}
\address[A.~Gilsbach]{
\newline%
Tokyo Institute of Technology \newline%
Department of Mathematics, School of Science, \newline%
2-12-1 Ookayama, Meguro-ku, Tokyo 152-8551, Japan}%
\email{gilsbach.a.aa@m.titech.ac.jp}%
\author{Philipp Reiter}
\address[Ph.~Reiter]{
\newline%
Chemnitz University of Technology,\newline%
Faculty of Mathematics,\newline%
09107 Chemnitz, Germany}
\email{reiter@math.tu-chemnitz.de}
\author{Heiko von der Mosel}
\address[H.~von~der~Mosel]{
\newline%
RWTH Aachen University,\newline%
Institut f\"ur Mathematik,\newline%
Templergraben 55,
52062 Aachen,
Germany}
\email{heiko@instmath.rwth-aachen.de}
\keywords{Bending energy,  tangent-point energy,  elastic knots,  symmetric criticality}

\date{May 18, 2021}%

\DeclareRobustCommand{\SkipTocEntry}[5]{}
\begin{document}

\maketitle

\begin{abstract}
 Minimizing the bending energy within knot classes
 leads to the concept of elastic knots which has been
 initiated in~\cite{vdm_1998}.
 Motivated by numerical experiments in~\cite{bartels-reiter_2018}
 we prescribe dihedral symmetry and establish existence
 of dihedrally symmetric elastic knots for
 knot classes admitting this type of symmetry. 
 Among other results we prove that the dihedral elastic
 trefoil is the union of two circles that form a (planar) figure-eight. 
 We also discuss some generalizations and limitations regarding other symmetries and knot classes.
\end{abstract}

\setcounter{tocdepth}1
\tableofcontents

\section{Introduction}\label{sec:intro}
The study of elastic knots was initiated by Gerlach et al.\ in 
\cite{gerlach-etal_2017}. Inspired by toy models of springy
knotted wires (see the images in \cite[Figure~7]{gerlach-etal_2017})
the existence of energy minimizing knotted 
configurations $\g_\vth$ has been established in any 
prescribed 
tame\footnote{\label{foot:tame} A knot class is {called} \emph{tame} if it contains a
polygonal representative \cite[Definition 1.3]{burde-zieschang_2003},
or equivalently, if and only if it contains a continuously differentiable
representative \cite[App.~I]{crowell-fox_1977}.}
knot class $\mathcal{K}$~\cite[Theorem 2.1]{gerlach-etal_2017}.
The total energy considered,
\begin{equation}\label{eq:total-energy}
E_\vth:= E+\vth \mathcal{R},\quad\Fo \vth >0,
\end{equation}
consists of the classic Euler--Bernoulli bending energy
\begin{equation}\label{eq:bending-energy}
E(\g):=\int_\g\kappa^2\,ds
\end{equation}
as the leading order term, together with a small multiple
of a repulsive potential $\mathcal{R}$ to 
avoid self-intersections. In order to analyse
the approximative shape of the minimizing knots $\g_\vth$ for 
small $\vth$ the authors study
the limit
$\vth\to 0$. {It is shown that the  minimizers $\g_\vth$ converge 
in $C^1$ to closed curves $\g_0$ that}
minimize the bending energy
\begin{equation*}%
E(\g_0)\le E(\beta)\quad\Foa \beta\in\CK,
\end{equation*}
where 
\begin{equation}\label{eq:admissible-curves}
\mathscr{C}(\mathcal{K}):=\{\g\in W^{2,2}(\R/\Z,\R^3):\mathscr{L}(\g)=1,\,
\abs{\g'}>0,\,
[\g]=\mathcal{K}\}.
\end{equation}
Here, $\R/\Z$ denotes the periodic interval of unit length.
These limiting curves $\g_0$ are called \emph{elastic knots 
for $\mathcal{K}$} according to \cite[Definition 2.3]{gerlach-etal_2017}, 
although they are not embedded unless $\mathcal{K}$ is
trivial; see \cite[Proposition 3.1]{gerlach-etal_2017}. 
One of the central results is the complete 
classification of elastic knots for all
torus knot classes $\TL(2,b)$ for odd $b\in\Z\setminus\{1,-1\}$.
\begin{theorem*}[{\cite[Corollary 6.5(i)]{gerlach-etal_2017}}]
The elastic torus knot $\g_0$ for $\TL(2,b)$ for any odd $b\in\Z\setminus
\{1,-1\}$ is the doubly covered circle.
\end{theorem*}
This result is confirmed by mechanical experiments with thin elastic
knotted
wires, as well as by various numerical simulations performed by
different groups of researchers as documented
in  \cite[Introduction p.~94]{gerlach-etal_2017}.
Also the more recent work on the corresponding
numerical gradient flow of Bartels et al.~\cite{bartels-etal_2018,bartels-reiter_2018} provides strong numerical evidence for
the doubly covered circle as the only possible elastic knot
for $\TL(2,b)$; see the left of Figure~\ref{fig:bartels-reiter}.

\subsection*{Symmetric configurations}
{Sometimes, however, this numerical gradient flow produces
a different limiting configuration exhibiting a dihedral symmetry
as depicted on the right of Figure~\ref{fig:bartels-reiter}.
This indicates the presence of a dihedrally symmetric critical point\footnote{This symmetric knot might be a saddle point, since there exist -- as reported 
in~\cite{bartels-reiter_2018} -- symmetry breaking perturbations with smaller
energy. A more systematic numerical investigation is under way to produce
more evidence of the nature of this critical point.}
of the total energy $E_\vth$.     }

\begin{figure}[!t]
\begin{center}
\includegraphics[scale=.5,trim=40 130 40 100,clip]{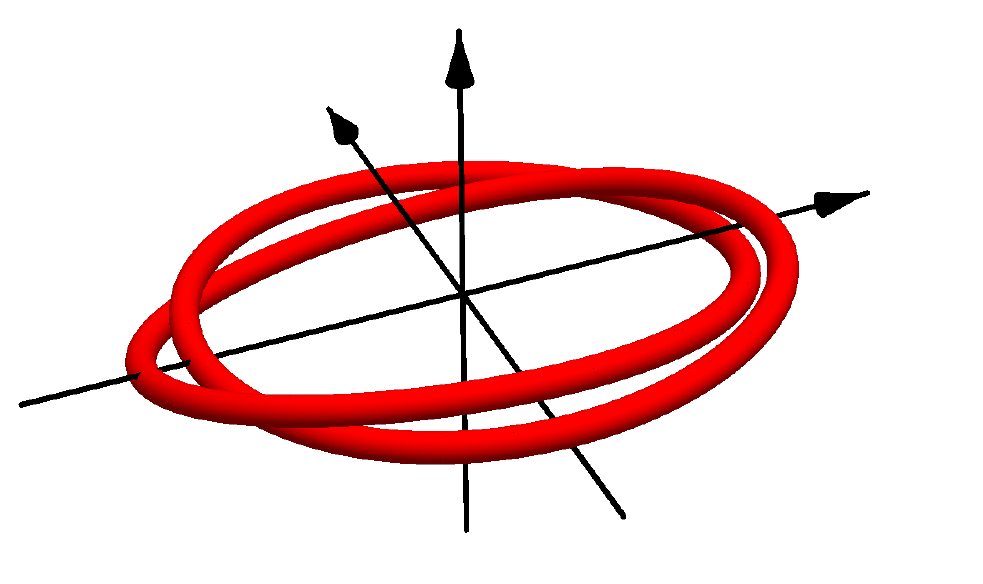}
\includegraphics[scale=.6,trim=20 70 30 50,clip]{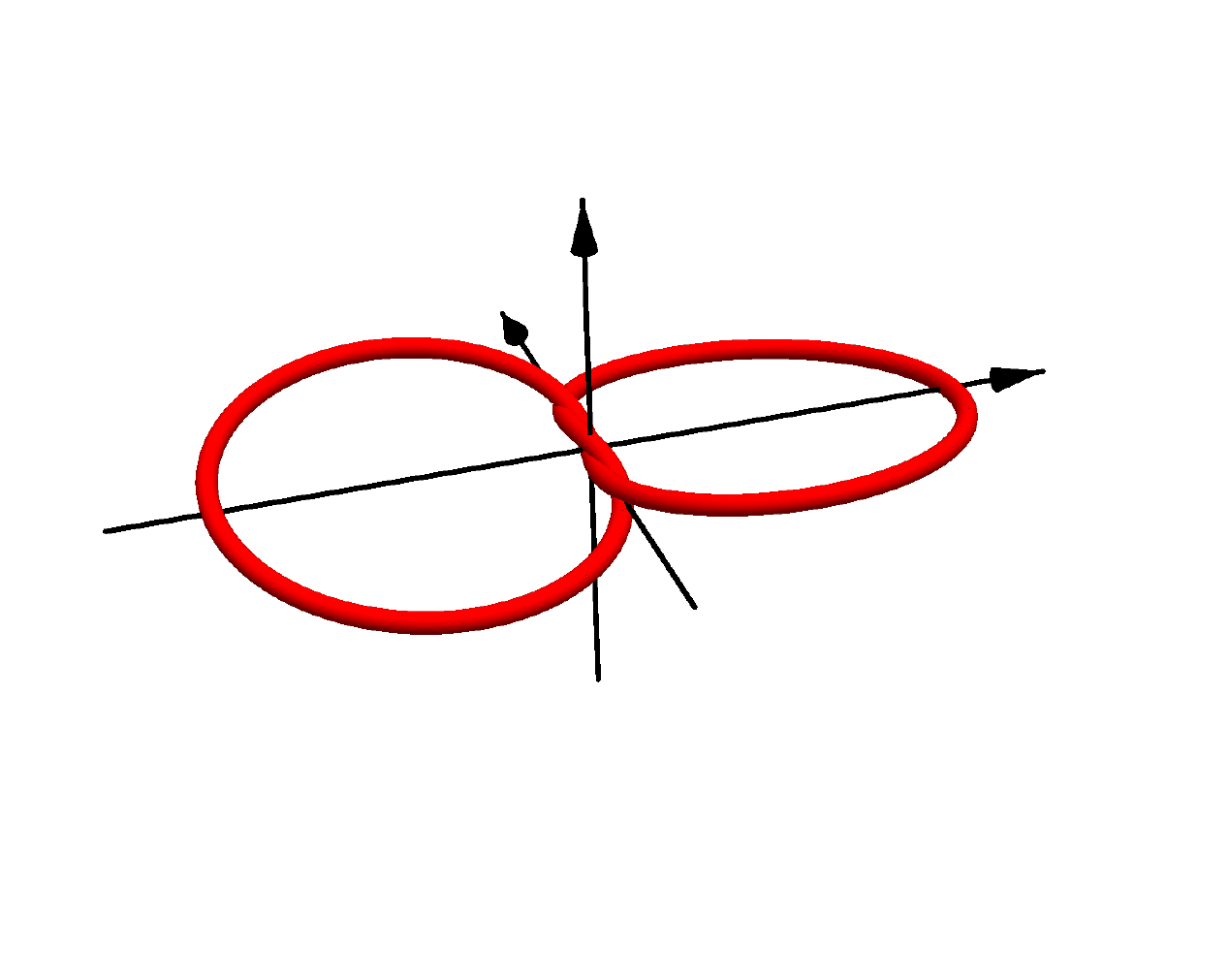}\\
\includegraphics[scale=.25,trim=0 0 -50 0,clip
]{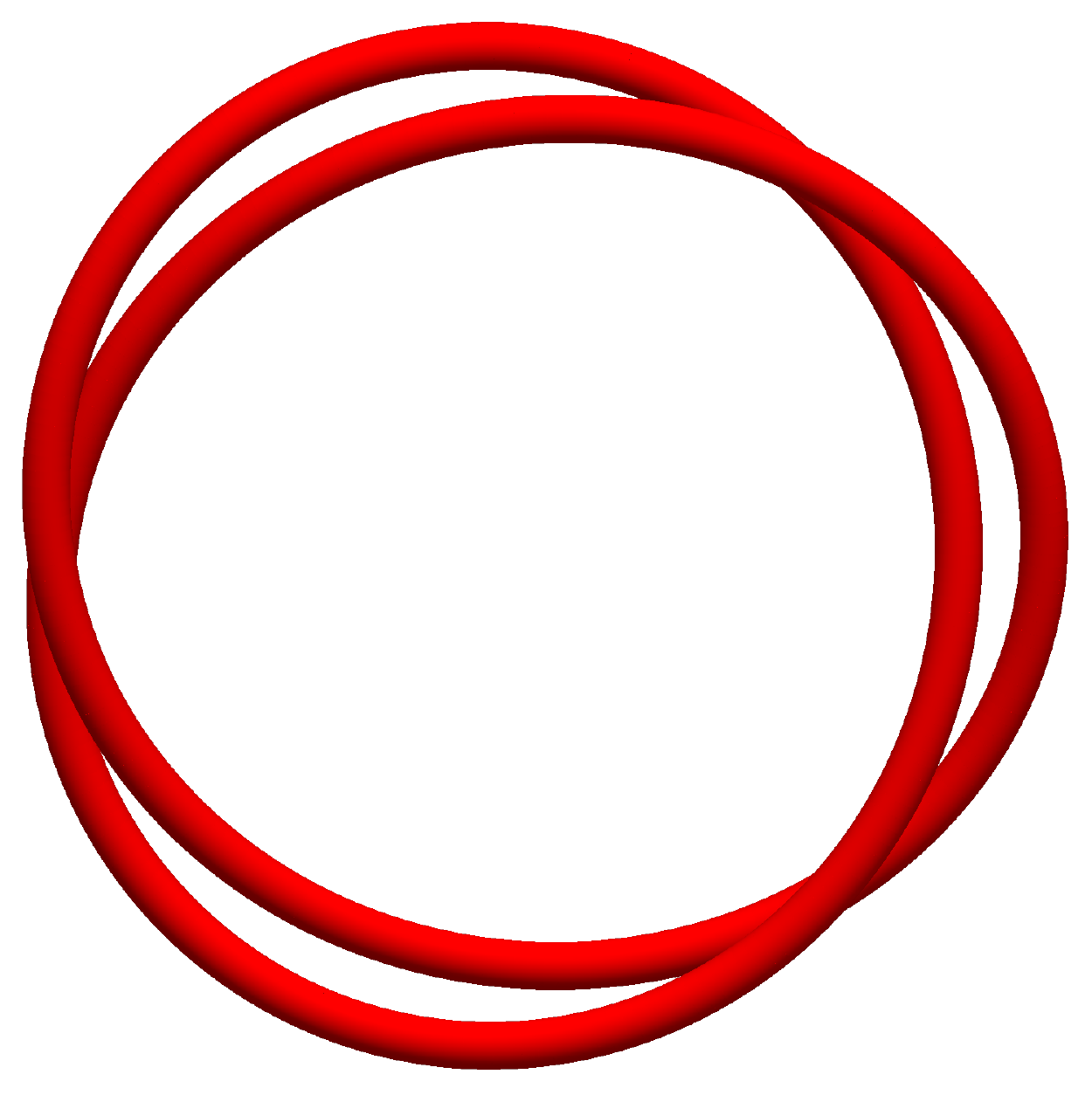}\qquad
\includegraphics[scale=.40,trim=-50 0 0 0,clip
]{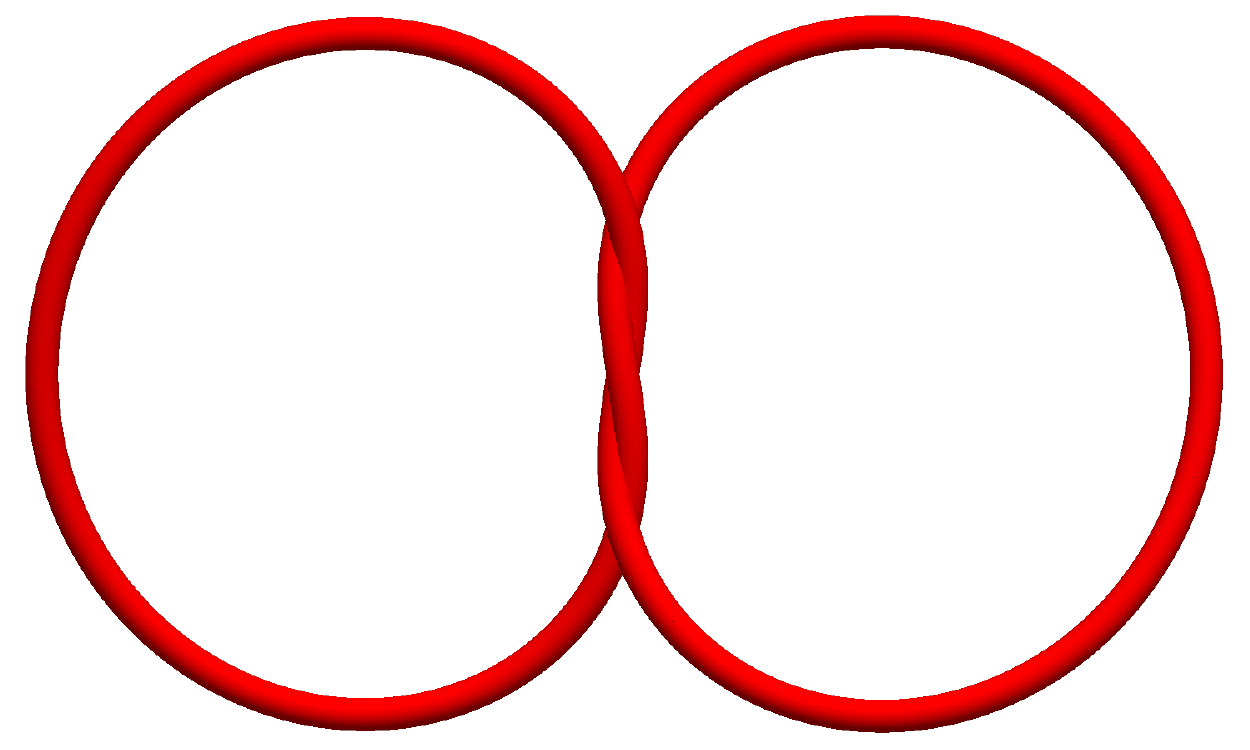}
\end{center}
\caption{%
Final states of the numerical gradient flow of Bartels et al.\@ \cite{bartels-etal_2018} for {the total energy} $E_\vth$ {for small $\vth >0$}
 in the case of the trefoil knot class. Left: The doubly covered circle is the usual limit
configuration as predicted by \cite{gerlach-etal_2017}.
Right: There are a few initial configurations that move towards 
an almost flat torus knot with dihedral symmetry under the gradient flow. %
\label{fig:bartels-reiter}}
\end{figure}

It is the purpose of the present paper to analytically support these
infrequent but reproducible
numerical observations. Namely, we use the principle of symmetric criticality
of  Palais \cite{palais_1979} to prove the existence of symmetric
critical points for the constrained variational problem
\begin{equation}\label{eq:P_theta}\tag{P${}_\vartheta$}
\textnormal{Minimize the total energy $E_\vartheta$ on the set $\mathscr{C}
(\mathcal{K})$.}
\end{equation}
Palais's principle, however, requires energy functionals of class $C^1$. To meet this precondition
(and to avoid the technical issues {connected}
with an alternative nonsmooth variant~\cite{cantarella-etal_2014a} 
of this principle)
we replace the nonsmooth ropelength functional
used in~\cite{gerlach-etal_2017} by 
a suitable power of the \emph{tangent-point energy} $\TP_q$.
It is a self-avoiding energy defined
on absolutely continuous regular closed curves $\g:\R/\Z\to\R$ as
\begin{equation}\label{eq:tan-point}
\mathcal{R}(\g):=\TP_q^{\frac1{q-2}}(\g):=\left(\int_{\R/\Z}\int_{\R/\Z}
\frac{1}{r^q_\textnormal{tp}(\g(s),\g(t))}\abs{\g'(s)}\abs{\g'(t)}\,dsdt
\right)^{\frac1{q-2}},
\end{equation}
where we restrict  to exponents  $q\in (2,4]$.
Here, $r_\textnormal{tp}(\g(s),\g(t))$ stands for the radius of the unique circle
through the points $\g(s)$ and $\g(t)$ that is tangent to $\g$ at $\g(s)$.
This energy was first suggested by
Buck and Orloff~\cite{buck-orloff} in the case $q=2$
and for general $q>2$ by Gonzalez and Maddocks in
\cite[p. 4773]{gonzalez-maddocks_1999}, and it was investigated analytically in detail
in~\cite{strzelecki-vdm_2012},
\cite{blatt_2013b}, and~\cite{blatt-reiter_2015a}.
Note that the Sobolev space $W^{2,2}$
continuously embeds into the fractional Sobolev space
$W^{2-(1/q),q}$ for $q\in (2,4]$,
which provides the exact  regularity framework to guarantee a finite 
and  continuously differentiable tangent point energy 
\cite[Remark 3.1]{blatt-reiter_2015a},
\cite{wings_2018}, so that the
total energy $E_\vth$  is continuously differentiable
on the open subset $W^{2,2}_\textnormal{ir}(\R/\Z,\R^3)$ of
\underline injective %
\underline regular closed curves of class $W^{2,2}$.  Consequently,  Palais's principle of symmetric
criticality is 
applicable. 
Furthermore, a suitably discretized version  of $\TP_q$ for $q\in (2,4]$
was
used for $\mathcal{R}$ in 
the numerical gradient flow in \cite{bartels-etal_2018} and
\cite{bartels-reiter_2018}.

\subsection*{Existence results}

\begin{theorem}[Existence of symmetric critical knots]
\label{thm:symmetric-critical}
Given a knot class~$\mathcal{K}$, assume that there is at least one knot 
with dihedral symmetry contained in $\CK$.
Then for every $\vartheta>0$ there
exists an arclength parametrized 
knot $\Gamma_\vartheta$ of knot type $\mathcal{K}$
with dihedral symmetry
that is critical for the constrained minimization
problem \eqref{eq:P_theta}.
More precisely, we have 
\begin{equation}\label{eq:intro-ELG}
 DE_\vartheta(\Gamma_\vartheta)h+\lambda D\mathscr{L}(\Gamma_\vartheta)h=0\quad\Foa h\in W^{2,2}(\R/\Z,\R^3),
\end{equation}
where $\lambda:=E_\vartheta(\Gamma_\vartheta)$.
\end{theorem}

Here, the term \emph{dihedral symmetry}\footnote{{There
 are different definitions of a \emph{symmetry group} for knots in the literature,
 cf.~\cite{adams_2004,burde-zieschang_2003,gruenbaum-shephard_1985} and references therein.
 In the present
 paper we consider Euclidean symmetries
 of an actual {space} curve.}} or, synonymously, \emph{$D_2$-symmetry}
refers to the action of the classic dihedral
group
$D_2$
(see Table~\ref{tab:dihedral-table}) on parametrized space curves  by rotating the curve's 
image by an angle of $\pi$ about any of the three coordinate axes combined
with an appropriate (dihedral) transformation of the periodic 
domain $\R/\Z$ of the curve. 
All this is made precise in Section 
\ref{sec:group}; see the examples in Figures~\ref{fig:tpc_pi} 
 and~\ref{fig:constrB}
for a preliminary impression of $D_2$-symmetric curves. In particular,
Figure~\ref{fig:constrB} depicts a dihedrally
symmetric torus knot of class $\TL(2,5)$ constructed in
Example~\ref{ex:D2-symmetric-torus-knot} with  a method which works for any
odd $b\in\Z\setminus\{1,-1\}$, so that the torus knot classes
$\TL(2,b)$ satisfy the hypothesis of Theorem~\ref{thm:symmetric-critical}.
\begin{corollary}[Existence of symmetric critical torus knots]
\label{cor:symmetric-critical}
Let $b\in\Z\setminus\{-1,1\}$ be odd. Then for
every $\vartheta>0$ there exists an arclength parametrized torus knot
$\Gamma_\vartheta
\in\mathscr{C}(\mathcal{T}(2,b))$ with
dihedral symmetry, which is critical for the constrained minimization
problem \eqref{eq:P_theta} for $\mathcal{K}=\TL(2,b)$.
\end{corollary}
Apart from the $D_2$-symmetry these existence results contain no
information about the actual shape of the critical knots $\Gamma_\vth$. 
In fact, there is a large variety of possible shapes of dihedrally symmetric curves.
In Lemma
\ref{lem:glueing} we provide a general mechanism how to construct
space curves with dihedral symmetry from just one arc satisfying rather mild
conditions on its endpoints.  
To obtain more specific
information on the shape of the symmetric critical points $\Gamma_\vth$
obtained in 
Theorem~\ref{thm:symmetric-critical} and 
Corollary~\ref{cor:symmetric-critical} it seems hard to exploit the
variational equation~\eqref{eq:intro-ELG} because of the complicated
differential
$D\TP_q$ of the non-local tangent-point energy $\TP_q$ as part of
the total energy $E_\vth$.  Following the
idea in \cite{gerlach-etal_2017} we study 
the limit $\vth\to 0$ instead, to
obtain limit configurations  whose shape {can} then be analyzed more easily to
yield the approximative shape of the $D_2$-symmetric critical knots $\Gamma_\vth$
for small $\vth>0$.
\begin{theorem}[Existence of symmetric elastic knots]
\label{thm:symmetric-elastic-knots}
Let  $\mathcal{K}$ be a fixed knot class that contains
a $D_2$-symmetric representative in $\CK$, and consider a 
sequence $\vartheta_j\to 0$ and the corresponding
$\textnormal{(P${}_{\vth_j}$)}$-critical $D_2$-symmetric
knots $\Gamma_{\vartheta_j}$ obtained in Theorem 
\ref{thm:symmetric-critical}.
Then there exists an arclength parametrized curve
$\Gamma_0\in W^{2,2}(\R/\Z,\R^3)$ with dihedral symmetry, and
a subsequence $\big(\Gamma_{\vartheta_{j_k}}\big)_k\subset
\br{\Gamma_{\vartheta_j}}_j$ such that the
$\Gamma_{\vartheta_{j_k}}$ converge weakly in $W^{2,2}$ and strongly
in $C^1$ to $\Gamma_0$ as $k\to\infty$. Moreover,
\begin{equation}\label{eq:E_b-symm-minimizer}
E(\Gamma_0)\le E(\beta)\quad\textnormal{for all $D_2$-symmetric 
$\beta\in\CK.$}
\end{equation}
\end{theorem}

\begin{definition}[Symmetric elastic knots]
\label{def:symmetric-elastic-knots}
Any such curve $\Gamma_0$ obtained in Theorem~\ref{thm:symmetric-elastic-knots}
is called a \emph{dihedral (or $D_2$-) elastic knot for $\mathcal{K}$.}
\end{definition}

\subsection*{Shapes of symmetric elastic knots}
The unknot class 
and the torus knot class 
$\TL(2,b)$
for any odd $b\in\Z\setminus\{1,-1\}$ satisfy the hypothesis of 
Theorem~\ref{thm:symmetric-elastic-knots}; see Examples~\ref{ex:one-circle}
and~\ref{ex:D2-symmetric-torus-knot}. Consequently, there are 
$D_2$-elastic knots for the unknot class and {for}
$\TL(2,b)$, and
we can determine their shapes, which also turn out to \emph{characterize}
these knot classes.

\begin{theorem}[The $D_{2}$-elastic unknot]
\label{thm:elasym-unknot}
{Up to reparametrization {and isometry}}, 
the only $D_2$-elastic unknot is the
once covered 
circle of length one.
Moreover,
if a $D_2$-elastic knot for some knot class $\mathcal{K}$ is the
once covered circle, then $\mathcal{K}$ is the unknot class.
{{Only the unknot class $\mathcal{K}$ satisfies}
\begin{equation}\label{eq:infimum-minimal-unknot}
\inf_{\beta\in\CK\atop
\textnormal{$\beta$ is $D_2$-symmetric}}E(\beta)
=\inf_{\beta\in W^{2,2}\atop\mathscr L(\beta)=1}E(\beta)=(2\pi)^2.
\end{equation}} 
\end{theorem}

If we have specific information about the infimal
bending energy on non-trivial
knots with
dihedral symmetry, then we can identify the shape of the corresponding
$D_2$-elastic knot. To make this more precise, recall that the natural lower bound for the total curvature
$\TC(\g):=\int_\g\kappa \,ds$
{is $2\pi$
by virtue of Fenchel's theorem~\cite{fenchel_1929}.
Applying H\"older's inequality it transfers to
the natural lower bound $(2\pi)^2$ for the bending energy
of closed curves of length one. 
{Therefore,} \eqref{eq:infimum-minimal-unknot}
is in fact equivalent to
$\inf\sett{E(\beta)}{\beta\in\CK,\beta\textnormal{ is $D_2$-symmetric}}\le(2\pi)^{2}$.
In case of non-trivial knots, the total curvature is bounded
below by $4\pi$}
according to the famous result of F\'ary and Milnor (see~\cite{fary_1949,milnor_1950})
which gives rise to the lower bound $(4\pi)^{2}$ for the bending
energy
for any non-trivial knot class $\mathcal{K}$.

Knot classes $\mathcal{K}$
for which the infimal bending energy \emph{equals} the
natural lower bound $(4\pi)^2$, i.e., for which
\begin{equation}\label{eq:infimum-minimal}
\inf_{\beta\in\CK\atop \beta\,\,\textnormal{is $D_2$-symmetric}}E(\beta)=(4\pi)^2
=\inf_{\CK}E(\cdot)
\end{equation}
are of particular interest, since they
provide the variational problem with
a high degree of rigidity; see Theorem~\ref{thm:rigidity}.

In Section~\ref{sec:example} we analyze a sequence of specific $D_2$-symmetric
$(2,b)$-torus knots like 
in Figure~\ref{fig:constrB} to show that all torus knot classes
$\TL(2,b)$ for odd integers $b\in\/\setminus\{-1,1\}$ satisfy
condition \eqref{eq:infimum-minimal}. It actually turns out that
there are no other knot classes satisfying \eqref{eq:infimum-minimal}. This
leads to the following central \emph{characterization of $D_2$-elastic
$(2,b)$-torus knots.}

\begin{theorem}[$D_2$-elastic $(2,b)$-torus knots]\label{thm:mainthm}
The following statements hold up to {isometry and} reparametrization. 
\begin{enumerate}
\item[\rm (i)]
The unique $D_2$-elastic $(2,b)$-torus knot for any odd 
$b\in\Z\setminus\{1,-1\}$ is 
the  
tangential pair 
of co-planar circles with exactly one point in common, denoted by $\tpc[\pi]$.
Any sequence of $D_2$-symmetric
$E_\vth$-critical $(2,b)$-torus
knots $\Gamma_\vth$ converges \emph{strongly}
in $W^{2,2}$ to $\tpc[\pi]$
 as $\vth\to 0$.
\item[\rm (ii)]
If a $D_2$-elastic knot for some knot class $\mathcal{K}$ is
$\tpc[\pi]$ then $\mathcal{K}=\TL(2,b)$
for some odd $b\in\Z\setminus\{1,-1\}$.
\item[\rm (iii)] 
If any {non-trivial} knot class $ \mathcal{K}$ satisfies 
\begin{equation}\tag{\ref{eq:infimum-minimal}*}
\inf_{\beta\in\CK\atop \beta\,\,\textnormal{is $D_2$-symmetric}}E(\beta)\le(4\pi)^2
\end{equation}
then $\mathcal{K}=\TL(2,b)$ for some odd $b\in\Z\setminus\{1,-1\}.$
\end{enumerate}
\end{theorem}
{Note that~\eqref{eq:infimum-minimal}
and~(\ref{eq:infimum-minimal}*) are in fact equivalent
due to the F\'ary--Milnor theorem.}

The one-parameter family of \emph{tangential pairs of circles} $\tpc[\varphi]$
for $\varphi\in [0,\pi]$ 
was introduced in \cite{gerlach-etal_2017}; see Figure~\ref{fig:tpc_phi}. It consists of (isometric
images of) pairs of circles each with radius $1/(4\pi)$ that 
intersect each other tangentially in at least one point. The parameter  
$\varphi$ describes the angle between the two planes spanned by the two
circles. Only for $\varphi=0$ and $\varphi=\pi$ the two planes coincide, and
the tangential pair of co-planar circles addressed in the rigidity
result, Theorem~\ref{thm:rigidity}, is {(an isometric image of)}  
$\tpc[\pi]$;
see Example~\ref{ex:tpc_pi}. Part (i) of Theorem~\ref{thm:mainthm} improves the
weak $W^{2,2}$-subconvergence of  $D_2$-symmetric $E_\vth$-critical
knots $\Gamma_\vth$, established in Theorem~\ref{thm:symmetric-elastic-knots} for 
general knot classes $\mathcal{K}$, now to the \emph{strong} convergence of
\emph{every} sequence
of $D_2$-symmetric $E_\vth$-critical points $\Gamma_\vth$
  to $\tpc[\pi]$ 
  for all torus knot classes $\TL(2,b)$. Therefore, the limit curve $\tpc[\pi]$  describes the approximate shape of the $\Gamma_\vth$ for small
  $\vth$, thus supporting the sometimes experimentally observed final configurations
  of the numerical gradient flow of Bartels et al.~\cite{bartels-etal_2018}; see Figure~\ref{fig:bartels-reiter}~{(right)}.
However, as pointed out before, it 
{seems unlikely} that these $D_2$-symmetric critical points
  $\Gamma_\vth$ are local minimizers of $E_\vth$. To clarify this, one would need to analyze
  the second variation of the total energy containing the complicated non-local terms 
  of the tangent-point energy $\TP_q$.

\begin{figure}
 \includegraphics[scale=.5,trim=60 130 40 50,clip]{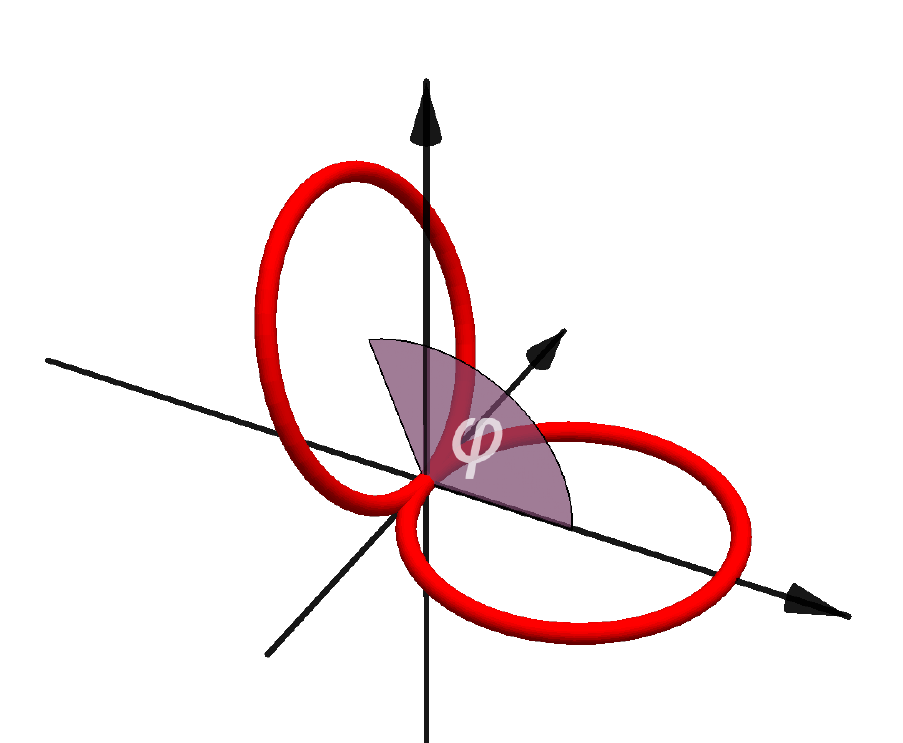}
  \caption{The family of tangential pairs of circles $\tpc$ parametrized by
  the opening angle $\varphi$ between the two planes containing the circles.}\label{fig:tpc_phi}
    \end{figure}

\subsection*{Outline}

The paper is structured as follows. In Section~\ref{sec:psk} we briefly review the
basics of the principle of symmetric criticality along the lines of the presentation
in~\cite{gilsbach_2018} and~\cite[Section 2]{gilsbach-vdm_2018}, from which  we adopted our approach to 
apply symmetric criticality to knotted space curves. The relevant facts about the
tangent-point energy  are presented in Section~\ref{sec:tanpoint}.

In Section~\ref{sec:group} we discuss the group action in detail, 
show how to construct {and characterize}
dihedrally symmetric space curves
with increasing regularity {(Definition~\ref{def:dihedral} 
and Proposition
\ref{prop:character-symmetry})}.  
{Of particular importance in our context are}
dihedrally symmetric circles and planar tangential pairs of circles 
{whose
exact location in space is determined in~Corollary~\ref{cor:circle}.}
We also prove in Lemma~\ref{lem:arclength-sym} that reparametrization to arclength
does not destroy the symmetry and provide a sharp a priori estimate on the
size of $D_2$-symmetric curves; see Lemma~\ref{lem:improved-Linfty-bound}.
We finally identify the suitable Banach manifold of $W^{2,2}$-regular knots
in Lemma~\ref{lem:Omega-banach}, on which the group $D_2$ acts in a sufficiently
regular way required by Palais's symmetric criticality principle 
(Lemma~\ref{lem:group-action-Omega}).

Section~\ref{sec:existence}
is devoted to proving the existence of the $D_2$-symmetric critical points 
as stated in Theorem~\ref{thm:symmetric-critical}, by
first minimizing a rescaled total energy on symmetric knots (Theorem
\ref{thm:existence-total}). By symmetric criticality these minimizers turn out
to be critical points among all knots in the given knot class 
 (Corollary~\ref{cor:symm-critic}) satisfying the desired Euler--Lagrange equation \eqref{eq:intro-ELG}
as shown in Corollary~\ref{cor:ELG}. Moreover, the existence of $D_2$-symmetric
elastic knots in any given tame knot class, i.e., the proof of Theorem
\ref{thm:symmetric-elastic-knots}, is established.
The remainder of Section
\ref{sec:existence} and Section~\ref{sec:example} deal with the shape of symmetric elastic knots, that is,
the proofs of Theorems~\ref{thm:elasym-unknot} and~\ref{thm:mainthm},
the latter with the help of 
a general rigidity result (Theorem~\ref{thm:rigidity}) for all knot classes
satisfying \eqref{eq:infimum-minimal} and an explicit convergence proof
of the torus knots of Example~\ref{ex:D2-symmetric-torus-knot} towards $\tpc[\pi]$ carried
out in Lemma~\ref{lem:convergence}.

In Section \ref{sec:discussion} we briefly touch on higher regularity
of the symmetric critical knots obtained in Theorem 
\ref{thm:symmetric-critical},
as well as on the question whether symmetric elastic knots are embedded.
The concept of symmetric knots also applies to
other symmetry classes than~$D_{2}$.
We give a {brief} outlook on the general case {in that section.}

The simulations shown in Figures~\ref{fig:bartels-reiter}
and~\ref{fig:threefoil} have been
carried out using the algorithm described in~\cite{bartels-reiter_2018}
which bases on an earlier work by Bartels~\cite{bartels_2013}.
We also refer to the app KNOTevolve~\cite{knotevolve}.

\section{Preliminaries}\label{sec:prelim}
\subsection{Principle of symmetric criticality}\label{sec:psk}
Let us briefly recall the notion of a group action on a Banach manifold modeled over
a Banach space  to describe symmetry in a mathematically rigorous 
way, cf. \cite[pp.~19--20, 26]{palais_1979}.
\begin{definition}\label{def:symmetry}
For $k\in\N$ let $\mathscr{M}$\! be a $C^k$-Banach manifold modeled
over a Banach space $\mathscr{B}$, and suppose that $(G,\circ)$ is a group.
\begin{enumerate}
\item[\rm (i)]
If there is a mapping $\tau$ assigning {to} each $(g,x)\in G\times\mathscr{M}$ {a point}
$\tau_g(x)\in\mathscr{M}$ such that
\begin{equation}\label{eq:homo}
\tau_{g\circ h}(x)=\tau_g(\tau_h(x))\Foa g,h,\in G, \,x\in\mathscr{M},
\end{equation}
then the group $G$ is said to \emph{act on $\mathscr{M}$\!}, and $\tau$
is called a \emph{representation of $G$ on $\mathscr{M}$\!.}
\item[\rm (ii)]
If for each $g\in G$ the mapping $\tau_g:\mathscr{M}\to\mathscr{M}$ is a
$C^k$-diffeomorphism, then $\mathscr{M}$\! is called a \emph{$G$-manifold
(of class $C^k$)}. For an infinite Lie group $G$ one additionally requires that
the representation $\tau$ is of class $C^k$ on $G\times
\mathscr{M}$.  If $\mathscr{M}$\! is itself a Banach space and $\tau_g$ is
linear then $\mathscr{M}$\! is said to be a \emph{$G$-space}.
\item[\rm (iii)]
For a $G$-manifold  $\mathscr{M}$\! the  \emph{$G$-symmetric subset}
$\Sigma\subset\mathscr{M}$ is defined as
\begin{equation}\label{eq:symmetric-subset}
\Sigma:=\{x\in\mathscr{M}:\tau_g(x)=x\quad\textnormal{for all $g\in G$}\}.
\end{equation}
\item[\rm (iv)]
A function $F:\mathscr{M}\to\R$ is called \emph{$G$-invariant} if and only if
\begin{equation}\label{eq:invariance}
F(\tau_g(x))=F(x)\quad\Foa g\in G,\,x\in\mathscr{M}.
\end{equation}
\end{enumerate}
\end{definition}
Palais proved the following \emph{symmetric criticality principle.}
\begin{theorem*}[{\cite[Theorem 5.4]{palais_1979}}]
Suppose $G$ is a compact Lie group and $\mathscr{M}$\! a $G$-manifold of class
$C^1$ over the Banach space $\mathscr{B}$ with the non-empty $G$-symmetric 
subset $\Sigma\subset\mathscr{M}$, and let $F:\mathscr{M}\to\R$ be a
$G$-invariant function of class $C^1$. Then $\Sigma$ is a $C^1$-submanifold
of $\mathscr{M}$\!, and $x\in\Sigma$ is a critical point of $F$ if and only if $x$
is critical for the restricted functional $F|_\Sigma:\Sigma\to\R$. {That is,
 if $D(F|_\Sigma)(x)v=0$ for all
$v\in T_x\Sigma$, then also $DF(x)w=0$ for all $w\in T_x\mathscr{M}.$}
\end{theorem*}

As every finite group is a Lie group, cf., e.g.,~\cite[p.~48, Example 5]{cohn_1957}, we infer
the following result that we apply in Section~\ref{sec:existence} to
obtain dihedrally symmetric critical knots for the total energy $E_\vth$.
\begin{corollary}[Symmetric criticality for finite groups]
\label{cor:psk}
If $G$ is a finite group and $\mathscr{M}$\! a $G$-manifold of class $C^1$ with
non-empty
$G$-symmetric subset $\Sigma\subset\mathscr{M}$, and if $F\in C^1(\mathscr{M})$
is  $G$-invariant, then any critical point of $F|_\Sigma$ is
also a critical point of $F$. 
\end{corollary}
Our choice of a Banach manifold will simply be an open subset $\Omega$ of
a Banach space $\mathscr{B}$. This allows us to identify the
differential of the energy $F:\Omega\to\R$ with the classic
Fr\'echet-differential $DF(x):T_x\Omega\simeq\mathscr{B}\to
T_{F(x)}\R\simeq\R,$ which can be computed by means of the first variation
\[
DF(x)h=\delta F(x,h):=\lim_{\epsilon\to 0}{\epsilon}^{-1}\big[F(x+\epsilon h)-F(x)\big]
\quad\Fo h\in\mathscr{B}.
\]

\subsection{Tangent-point energy}\label{sec:tanpoint}
As mentioned in the introduction, Gonzalez and Maddocks suggested in 
\cite[p. 4773]{gonzalez-maddocks_1999} to consider the tangent-point energy 
\eqref{eq:tan-point} as a candidate for a valuable knot energy. This was confirmed
in the work of  P. Strzelecki and the third author
\cite{strzelecki-vdm_2012} starting at a rather low level of regularity
with just rectifiable curves. In fact,  arclength parametrizations $\Gamma
\in C^{0,1}(\R/\Z,\R^3)$ of
rectifiable curves with finite tangent-point energy $\TP_q(\Gamma)<\infty$ for 
$q\ge 2$ are either injective or they are multiple coverings
of one-dimensional manifolds \cite[Theorem  1.1]{strzelecki-vdm_2012}.
{In addition, such curves $\Gamma$} are
of class $C^{1,1-(2/q)}$ if $q>2$; see \cite[Theorem 1.3]{strzelecki-vdm_2012}.

In the present context, however, dealing with the bending energy $E$ we
start  at the already higher regularity level
of closed $W^{2,2}$-curves which -- according to the Morrey--Sobolev embedding
theorem -- are automatically of class $C^{1,1/2}$.
Consequently, it suffices to
review Blatt's regularity results~\cite{blatt_2013b,blatt-reiter_2015a}
on $C^1$-curves with finite tangent-point energy.\footnote{\label{foot:decoup}Notice that the two-parameter family of energies 
$\TP^{(p,q)}$ considered in \cite{blatt-reiter_2015a} contains
the tangent-point energy, more precisely
$\TP_{q}=2^{q}\TP^{(2q,q)}$.}
One of the central results \cite[Theorem~1.1]{blatt_2013b,blatt-reiter_2015a}
characterizes finite energy among embedded curves
by fractional Sobolev regularity $W^{2-(1/q),q}$.
Here, we only need
 one part of that statement explicitly, in fact,  in a slightly sharpened version for not
 necessarily arclength parametrized curves established in
 \cite[Theorem~3.2~(ii)]{gilsbach-vdm_2018}.\footnote{{For the proof of}
 \cite[Lemma~A.1]{gilsbach-vdm_2018}, which is used {to establish}
 \cite[Theorem~3.2~(ii)]{gilsbach-vdm_2018}, {see the updated}
  arXiv version.}

\begin{theorem}\label{thm:energy-space}
Let $q\in (2,\infty)$ and suppose that $\g\in W^{2-(1/q),q}(\R/\Z,\R^3)$
is injective  and satisfies $|\g'|>0$ on $\R/\Z$ . Then $\TP_q(\g)<\infty$.
\end{theorem}
The other part of Blatt's characterization (or \cite[Theorem 1.3]{strzelecki-vdm_2012}
for that matter) can be  used to
quantify the degree of embeddedness
for arclength parametrized curves $\Gamma:\R/\Z\to\R^3$
by means of the \emph{bi-Lipschitz constant}
\begin{equation*}%
\BiLip(\Gamma):=\inf_{s,t\in\R/\Z\atop s\not=t}\frac{\abs{\Gamma(s)-\Gamma(t)}}{
\abs{s-t}_{\R/\Z}}.
\end{equation*}
\begin{lemma}[Bi-Lipschitz estimate for finite $\TP$-energy
{\cite[Proposition 2.7]{blatt-reiter_2015a}}]\label{lem:bilipschitz}
For any $q> 2$ and $T>0$ there is a constant $C=C(q,T)>0 $ such that any arclength parametrized and  injective curve $\Gamma\in C^{0,1}(\R/\Z,\R^3)$ with
$\TP_q(\Gamma)\le T$ satisfies
\begin{equation*}%
\BiLip(\Gamma)\ge C.
\end{equation*}
\end{lemma}
So far, we have reported on the  effects that finite tangent-point energy has
on the curve. Let us conclude this short review with continuity and regularity
properties of the energy itself.
\begin{theorem}[Regularity of the tangent-point energy]
\label{thm:reg-tanpoint}
Let $q>2$.  The tangent-point energy $\TP_q$ is sequentially
lower semicontinuous with respect to $C^1$-conver\-gence. Moreover,
$\TP_q$ is continuously differentiable on regular embedded closed curves
of fractional Sobolev regularity $W^{2-(1/q),q}$.
\end{theorem}
\begin{proof}
Lower semicontinuity of the tangent-point energy was shown in
\cite[p.~1513]{strzelecki-etal_2013a}, whereas  continuous differentiability
was verified in \cite{wings_2018} using the first variation formula in
\cite[Theorem 1.4]{blatt-reiter_2015a} and the line of arguments 
used for the corresponding regularity statement for integral Menger curvature
in \cite[Theorem~3]{blatt-reiter_2015b}.
\end{proof}

\subsection{Isotopy stability}

Several times in the proofs we will rely on the fact
that knot classes are stable with respect to $C^1$-perturbations.
Variants of the following statement can be found in~\cite{DEJvR,reiter_2005,blatt_2009,denne-sullivan_2008}.
\begin{lemma}[Ambient isotopy is open in $C^{1}$]\label{lem:isot}
 For any embedded $\gamma\in C^{1}(\R/\ell\Z,\R^{3})$
 there is an $\eps>0$ such that any $\tilde\gamma\in C^{1}(\R/\ell\Z,\R^{3})$
 with $\norm[C^{1}]{\tilde\gamma-\gamma}<\eps$
 is also embedded and belongs to the same knot class as $\gamma$.
\end{lemma}

\section{Group action on parametrized curves}\label{sec:group}

In order to describe the dihedral symmetry of parametrized
closed curves $\g:\R/\ell\Z\to\R^3$ we use two different representations of
the dihedral group $D_2:=\{d_0\equiv e,d_1,d_2,d_3\}$
with the multiplication table depicted in Table~\ref{tab:dihedral-table},
\begin{table}
\begin{tabular}{|c||c|c|c|c|}\hline
 & $e$ & $d_1 $ & $d_2$ & $d_3$ \\
 \hline
 \hline
 $e$ & $e$ & $d_1$ & $d_2$ & $d_3$ \\
 \hline
 $d_1$ & $d_1$ & $e$ & $d_3$ & $d_2$  \\
 \hline
 $d_2$ & $d_2$ & $d_3$ & $e$ & $d_1$ \\
 \hline
 $d_3$ & $d_3$ & $d_2$ & $d_1$ & $e$ \\
 \hline
 \end{tabular}

\medskip

\caption{The multiplication table of the dihedral group $D_2$.}
\label{tab:dihedral-table}

\end{table}
where $e$ denotes 
the identity element.  Namely, in view of the symmetry of the curves'
images we consider
the subgroup $\{\Id_{\R^3}\equiv R_0,R_1,R_2,R_3\}\subset SO(3)$
containing the rotations $R_i$ about the coordinate axes $\R\mathbf{e_{\boldsymbol{i}}}$
for $i=1,2,3,$ with rotational angle $\pi$, that is, written as matrices
with respect to the standard coordinate basis $\{\mathbf{e_1},
\mathbf{e_2},
\mathbf{e_3}\}\subset\R^3$,
\begin{equation}\label{eq:rotations}
R_1:=\begin{pmatrix}
  1 & 0 & 0 \\
  0 & -1 & 0 \\
  0 & 0 & -1
  \end{pmatrix},\,\,
R_2:=\begin{pmatrix}
  -1 & 0 & 0 \\
  0 & 1 & 0 \\
  0 & 0 & -1
  \end{pmatrix},\,\,
R_3:=\begin{pmatrix}
  -1 & 0 & 0 \\
  0 & -1 & 0 \\
  0 & 0 & 1 
  \end{pmatrix}.
\end{equation}
To take into account  the curves'  parametrizations, we use
in addition  the 
mappings $\psi_i^\ell:\R/\ell\Z\to \R/\ell\Z$ on the periodic
domain $\R/\ell\Z$, defined
as
\begin{equation}\label{eq:inner-action}
\psi^\ell_i(t):=
\begin{cases}
t\pmod\ell & \Fo i=0,\\
-t+\frac{\ell}2\pmod\ell & \Fo i=1,\\
t-\frac{\ell}2 \pmod\ell & \Fo i=2,\\
-t+\ell \pmod\ell & \Fo i=3.
\end{cases}
\end{equation}
It is easy to check that for all mutually distinct $i,j,k\in\{1,2,3\}$
the following identities hold.
\begin{align}
R_i\circ R_i=\Id_{\R^3} & \quad\AND\quad 
\psi^\ell_i\circ\psi^\ell_i=\Id_{\R/\ell\Z},
\label{eq:group-property1}
\\
R_i\circ R_j= R_k & \quad\AND\quad \psi^\ell_i\circ\psi^\ell_j=\psi^\ell_k,
\label{eq:group-property2}
\\
R_i|_{\R\mathbf{e_{\boldsymbol{i}}}}=\Id_{\R\mathbf{e_{\boldsymbol{i}}}} & \quad\AND\quad
  R_i|_{\R\mathbf{e_{\boldsymbol{k}}}}=R_j|_{\R\mathbf{e_{\boldsymbol{k}}}}
  =-\Id|_{\R\mathbf{e_{\boldsymbol{k}}}}.  \label{eq:group-property3}
\end{align}
Now we define how $D_2$ acts on the Banach space $C^0(\R/\Z,\R^3)$
of continuously parametrized closed curves (equipped with the
norm $\norm[C^0]{\cdot}$).
\begin{definition}\label{def:group-action}
Let $\tau^\ell:D_2\times C^0(\R/\ell\Z,{\R^{3}})\to C^0(\R/\ell\Z,\R^{3})$,
  mapping
$(d_i,\g)\mapsto\tau^\ell_{d_i}(\g)$ for $d_i\in D_2$, $i=0,1,2,3,$ and $\g\in
C^0(\R/\ell\Z,\R^{3})$, be given by
\begin{equation}\label{eq:group-action}
\tau^\ell_{d_i}(\g)(t):=R_i\circ\g\big(\psi^\ell_i(t)\big)\quad\Fo t\in\R/\ell\Z,\,\, i=0,1,2,3.
\end{equation}
\end{definition}
\begin{lemma}[$C^0(\R/\ell\Z,\R^{3})$ is a smooth $D_2$-space]
\label{lem:group-action-C0}
The mapping $\tau^\ell$ acts on $C^0(\R/\ell\Z,{\R^{3}})$, and under this action
$C^0(\R/\ell\Z,{\R^{3}})$ becomes a smooth $D_2$-space.
\end{lemma}

\begin{proof}
It is obvious that $\tau^\ell_{d_i}(\g)\in C^0(\R/\ell\Z,\R^3)$ for each
$i=0,1,2,3$, and $\g\in C^0(\R/\ell\Z,\R^3)$,
since any rotation in the image and affine linear transformation of
the periodic 
domain neither changes the $C^{0}$-regularity nor the $\ell$-periodicity.

According to the multiplication table of the group $D_2$ (see Table
\ref{tab:dihedral-table}) and by the properties 
\eqref{eq:group-property1}--\eqref{eq:group-property3}
we have 
\[
\tau^{\ell}_{d_i\circ d_i}(\g)(t)\overset{\textnormal{Table~\ref{tab:dihedral-table}}}{=}
\tau_e^{{\ell}}(\g)(t)\overset{\eqref{eq:group-action},\eqref{eq:inner-action}}{=}\g(t)\Foa t\in\R/\ell\Z,\,i=0,1,2,3.
\]
On the other hand, again by \eqref{eq:group-action} and \eqref{eq:group-property1},
\begin{align*}
\tau^{\ell}_{d_i}\big(\tau^{\ell}_{d_i}(\g)\big)(t) &
\overset{\eqref{eq:group-action}}{=}R_i\circ \big(\tau^{\ell}_{d_i}(\g)(\cdot)\big)
(\psi_i^{\ell}(t))\\
& \overset{\eqref{eq:group-action}}{=} 
R_i\circ R_i\circ\gamma(\psi_i^{\ell}(\psi^{\ell}_i(t)))
\overset{\eqref{eq:group-property1}}{=}\g(t)\Foa t\in\R/\ell\Z,\,i=0,1,2,3,
\end{align*}
which proves the homomorphism property \eqref{eq:homo} for identical
group elements in $D_2$. For $i\not= j=0$ there is nothing to prove
since $j=0$ corresponds to the identity elements in the respective
representations of $D_2$. For $i\not= j$, $i,j\in\{1,2,3\}$ we use
the multiplication rules in Table~\ref{tab:dihedral-table} and the 
definition \eqref{eq:group-action}
to find, on the one hand, 
\begin{equation}\label{eq:one-hand}
\tau^{\ell}_{d_i\circ d_j}(\g)(t)\overset{\textnormal{Table~\ref{tab:dihedral-table}}}{=}\tau^{\ell}_{d_k}(\g)
\overset{\eqref{eq:group-action}}{=}R_k\circ\g(\psi_k^{\ell}(t)),
\end{equation}
whereas  \eqref{eq:group-action}, as well as
\eqref{eq:group-property2}
lead to
\begin{align*}
\tau^{\ell}_{d_i}\big(\tau^{\ell}_{d_j}(\g)\big)(t)
& \overset{\eqref{eq:group-action}}{=}
\tau^{\ell}_{d_i}\big(R_j\circ\g(\psi_j^{\ell}(\cdot))\big)(t)\\
& \overset{\eqref{eq:group-action}}{=}
R_i\circ R_j\circ\g(\psi_{ j}^{\ell}(\psi_{ i}^{\ell}(t)))
\overset{\eqref{eq:group-property2}}{=}
R_k\circ\g(\psi_k^{\ell}(t)),
\end{align*}
which equals the expression in \eqref{eq:one-hand}. We have shown so far
that $\tau$ is indeed a representation of $D_2$ on  $C^0(\R/\Z,\R^3)$. 
Since 
$
\tau^{\ell}_{d_i}(\sigma\g+\eta)(t)  = R_i\circ (\sigma\g +\eta)(\psi_i^{\ell}(t))
 = \sigma R_i\circ\g(\psi_i^{\ell}(t))+ R_i\circ\eta (\psi_i^{\ell}(t))
 = 
\sigma\tau^{\ell}_{d_i}(\g)+\tau^{\ell}_{d_i}(\eta)
$
for all $\g,\eta\in C^0(\R/\ell\Z,\R^3)$ and  $\sigma\in\R$,  one finds that
 $\tau^{\ell}_{d_i}:C^0(\R/\ell\Z,\R^3)\to C^0(\R/\ell\Z,\R^3)$ is
  linear for all $i=0,1,2,3,$ so that
 the Banach space $C^0(\R/\ell\Z,\R^3)$ is indeed a $D_2$-space 
 in the sense of Definition 
~\ref{def:symmetry}, part (ii).
\end{proof}
Symmetric curves are of particular interest here, {which are defined as follows.}
 
\begin{definition}[Dihedrally symmetric curves]\label{def:dihedral}
 A curve is called \emph{dihedrally symmetric} or \emph{$D_{2}$-symmetric} if it belongs to the \emph{$D_2$-symmetric set}
\begin{equation}\label{eq:Sigma}
\Sigma^\ell:=\{ \g\in C^0(\R/\ell\Z,\R^3): 0<\mathscr{L}(\g)<\infty,\,
\tau^{\ell}_{d_i}(\g)=\g\,\,\textnormal{for all $d_i\in D_2$}\}.
\end{equation}
\end{definition}

{Now we 
provide a method to systematically construct examples
of $D_2$-symmetric  curves (of finite length). For that we}
glue copies of an  open or closed arc $\alpha:[0,\ell/4]\to\R^3$ of finite
length together to obtain
a mapping $g:[0,\ell)\to\R^3$ as
\begin{equation}\label{eq:glueing}
g(t):=\begin{cases}
\alpha(t) & \Fo t\in [0,\ell/4),\\
R_1\circ \alpha(\psi^\ell_1(t)) & \Fo t\in [\ell/4,\ell/2),\\
R_2\circ\alpha(\psi^\ell_2(t)) & \Fo t\in [\ell/2,3\ell/4),\\
R_3\circ \alpha(\psi^\ell_3(t)) & \Fo t\in [3\ell/4, \ell),
\end{cases}
\end{equation}
and investigate first under which circumstances this glueing process
leads to a closed curve with a certain regularity. Notice that $\psi_i^\ell(
[i\ell/4,(i+1)\ell/4))=(0,\ell/4]$ for $i=1,3,$ and $\psi_2^\ell([\ell/2,
3\ell/4))=[0,\ell/4)$ by definition of the $\psi_j^\ell$ in
\eqref{eq:inner-action}, so that $g$ in \eqref{eq:glueing} is well-defined.
{It turns out that this construction does not only produce 
examples of $D_2$-symmetric curves but also \emph{characterizes} 
this symmetry\footnote{{Fixing the rotational axes with $\eqref{eq:rotations}$ and the
corresponding parameter transformations in \eqref{eq:inner-action} enforces
a rigidity on the class {of} $D_2$-symmetric curves, which is reflected
in the statements of Proposition \ref{prop:character-symmetry} and
Corollary \ref{cor:circle} below.}}.
}

\begin{proposition}\label{prop:character-symmetry}
{{\rm (i)\,}
$\g\in\Sigma^\ell$ if and only if there exists an arc
$\alpha\in C^0([0,\ell/4],\R^3)$ of
positive and finite length satisfying
\begin{equation}\label{eq:point-constraints}
\alpha(0)\in\R\mathbf{e_3}\quad\AND\quad \alpha(\ell/4)\in\R\mathbf{e_1},
\end{equation}
such that $\g$ coincides with the $\ell$-periodic extension of $g$ defined
in \eqref{eq:glueing}.
In particular, the points $\g(0)$ and $\g(\ell/2)$ are contained in $\R\mathbf{e_3}$,
whereas $\g(\ell/4)$ and $\g(3\ell/4)$ are contained in $\R\mathbf{e_1}$.}

{
{\rm (ii)\, }
$\g\in\Sigma^\ell\cap W^{1,p}(\R/\ell\Z,\R^3)$ for some $p\in [1,\infty]$
 if and only if $\g=g$, where 
$\alpha$ is of class $W^{1,p}$ and satisfies \eqref{eq:point-constraints}.
}

{
{\rm (iii)\,} 
$\g\in\Sigma^\ell\cap C^1(\R/\ell\Z,\R^3)$ if and only if $\g =g$ {in the sense of~\textup{(i)}}, where
$\alpha\in C^1([0,\ell/4],\R^3)$ satisfies \eqref{eq:point-constraints} and 
\begin{equation}\label{eq:tangent-constraints}
\alpha'(0)\in\spann\{\mathbf{e_1},\mathbf{e_2}\}\quad\AND\quad
\alpha'(\ell/4)\in\spann\{\mathbf{e_2},\mathbf{e_3}\}.
\end{equation}
Moreover, $\g\in\Sigma^\ell\cap W^{2,p}(\R/\ell\Z,\R^3)$ for some $p\in [1,\infty]$
 if and only if, in addition to the above properties,  $\alpha$ is of class $W^{2,p}$.
 }

{
{\rm (iv)\,}
Let $S_i$ be the reflection in the coordinate plane $\mathbf{e_{\boldsymbol{i}}^\perp}$ for  $i=1,2,3.$ Then, 
$S_i(\Sigma^\ell)\subset\Sigma^\ell$, $S_i(\Sigma^\ell\cap C^1)\subset
\Sigma^\ell\cap C^1$, and $S_i(\Sigma^\ell\cap W^{k,p})
\subset \Sigma^\ell\cap W^{k,p}$ for $p\in [1,\infty]$, $k=1,2$.}
\end{proposition}
{The proof of this proposition will follow from the following partial 
results.}

\begin{lemma}[Glueing produces closed curves]
\label{lem:glueing}
Suppose $\alpha\in C^0([0,\ell/4],\R^3)$ has length $\mathscr{L}(\alpha)\in (0,\infty)$, 
then the mapping $g$ defined according to
\eqref{eq:glueing} has length $\mathscr{L}(g)=
4\mathscr{L}(\alpha)$. Moreover, $g$ is closed and  continuous, 
i.e., of class $C^0(\R/\ell\Z,\R^3)$ 
if and only if \eqref{eq:point-constraints}.
Finally, {if $\alpha$  is continuously differentiable,} 
the curve $g$ is of class $C^1(\R/\ell\Z,\R^3)$ if and only if in addition
to  \eqref{eq:point-constraints} the tangents of $\alpha$ satisfy
\eqref{eq:tangent-constraints}.
\end{lemma}
\begin{proof}
Since $\alpha$ is continuous one has $\mathscr{L}_{[0,\ell/4)}(g)=
\mathscr{L}(\alpha)$ (see, e.g., 
\cite[VIII, Section 5, Theorem 1, p. 223]{natanson_2016}), 
and because the rotated images of $\alpha$ have the same length as
$\alpha$, the statement
about $\mathscr{L}(g)$ is immediate. 

{The $\ell$-periodic extension of}
the piecewise defined curve $g$ is continuous  if and only if the
following four identities hold true:
\begin{align}
g(0)=\lim_{t\nearrow \ell}R_3\circ \alpha(-t+\ell),\quad & g(\ell/4)=\lim_{t\nearrow \ell/4}\alpha(t),\label{eq:pt-cond1}\\
g(\ell/2)=\lim_{t\nearrow \ell/2}R_1\circ \alpha (-t+\ell/2),\quad & g(3\ell/4)=\lim_{t\nearrow 3\ell/4}R_2\circ \alpha(t-\ell/2),\label{eq:pt-cond2}
\end{align}
where we have already plugged in   the definition \eqref{eq:glueing} in the respective limits
on the right-hand sides. Using the continuity of $\alpha$ on the right-hand side we obtain
the  conditions $\alpha(0)=R_3\circ \alpha(0)$ and $\alpha(\ell/4)=R_1\circ \alpha(\ell/4)$ that are equivalent to \eqref{eq:pt-cond1}, and with the help of \eqref{eq:group-property1} and
\eqref{eq:group-property2} exactly the same conditions equivalent to
\eqref{eq:pt-cond2}. 
{From~\eqref{eq:group-property3} we infer
$\ker\br{\Id-R_{i}}=\R\mathbf{e_{\boldsymbol{i}}}$} for $i=1,2,3$,
{which} implies that these
identities on the endpoints $\alpha(0)$ and $\alpha(\ell/4)$ are equivalent to \eqref{eq:point-constraints}. 

{Provided that $\alpha$ is continuosly differentiable},
$C^1$-regularity of $g$ is equivalent to the pointwise conditions \eqref{eq:pt-cond1}
and \eqref{eq:pt-cond2} in combination with the tangential conditions
\begin{align}
g'(0)=\lim_{t\nearrow \ell}R_3\circ \big(- \alpha'(-t+{\ell})\big),\quad & g'(\ell/4)=\lim_{t\nearrow \ell/4}\alpha'(t),\label{eq:tan-cond1}\\
g'(\ell/2)=\lim_{t\nearrow \ell/2}R_1\circ \big(- \alpha' (-t+\ell/2)\big),\quad & g'(3\ell/4)=\lim_{t\nearrow 3\ell/4}R_2\circ \alpha'(t-\ell/2),\label{eq:tan-cond2}
\end{align}
where the minus signs  in the respective left equations are a 
consequence of the chain rule. Now by continuity of $\alpha'$ we obtain from \eqref{eq:tan-cond1}
the equivalent conditions $\alpha'(0)=-R_3\circ \alpha'(0)$ and $\alpha'(\ell/4)=-R_1\circ \alpha'(\ell/4)$.
{From $\ker\br{\Id+R_{i}}=\spann\set{\mathbf{e_{\boldsymbol{j}}},\mathbf{e_{\boldsymbol{k}}}}$ we deduce that they are}
equivalent to~\eqref{eq:tangent-constraints}. 
Exploiting~\eqref{eq:tan-cond2} again with the help of~\eqref{eq:group-property1}
and~\eqref{eq:group-property2} leads to the same conditions on $\alpha'(0)$ and
$\alpha'(\ell/4)$.
\end{proof}

From the Morrey--Sobolev embedding in one dimension together with 
well-known glueing properties for Sobolev functions 
\cite[E3.6 \& E3.7]{alt_2016} one readily obtains the following corollary.
\begin{corollary}[Glueing Sobolev arcs]
\label{cor:glueing}
If $\alpha\in W^{1,p}((0,\ell/4),\R^3)$ for any $p\in [1,\infty]$, 
then $g$ defined in \eqref{eq:glueing}
is a closed curve of class $W^{1,p}(\R/\ell\Z,\R^3)$ if and only if
the continuous representative of $\alpha$ satisfies \eqref{eq:point-constraints}.
If $\alpha\in W^{2,p}((0,\ell/4),\R^3)$, then $g$ is a closed curve of
class $W^{2,p}(\R/\ell\Z,\R^3)$ if and only if
the $C^1$-represen\-ta\-tive of $\alpha$ satisfies
\eqref{eq:point-constraints} and \eqref{eq:tangent-constraints}.
\end{corollary}

Now we are in the position to prove that the constructed curve $g$ in 
\eqref{eq:glueing} is also $D_2$-symmetric.
\begin{lemma}[$D_2$-symmetric curves]
\label{lem:D2-symmetric-curves}
If  $\alpha\in C^0([0,\ell/4],\R^3)$ has finite and positive length and satisfies
\eqref{eq:point-constraints}, then the curve 
$g$ defined in \eqref{eq:glueing} is $D_2$-symmetric, that is, $g\in\Sigma^\ell$.
If $\alpha$ is of class $W^{1,p}$ for some $p\in [1,\infty]$, then so is the $D_2$-symmetric
curve  $g$, and 
if $\alpha\in C^1([0,\ell/4],\R^3)$ or if $\alpha$ is of class $W^{2,p},$ and satisfies \eqref{eq:tangent-constraints} in addition, 
then $g$ is a $D_2$-symmetric closed curve of class $C^1$, or $W^{2,p},$
respectively.
\end{lemma}
\begin{proof}
The regularity statements follow from Lemma~\ref{lem:glueing} and Corollary
\ref{cor:glueing}, so it suffices to prove $D_2$-symmetry,  for which we merely need to
show that 
\begin{equation}\label{eq:symmetry-check}
\tau^\ell_{d_i}(g)=g\quad\Fo i=1,2,3.
\end{equation}
We  only treat the case $i=1$ in full detail, the cases $i=2,3$ are very similar.

For $t\in [0,\ell/4)$ we have $\psi_1^{\ell}(t)\in (\ell/4,\ell/2]$, and therefore
by definition of $g$
\begin{align*}
\tau^\ell_{d_1}(g)(t) & \overset{\eqref{eq:group-action}}{=} R_1\circ 
g(\psi^\ell_1(t))
\overset{\eqref{eq:glueing}}{=}
R_1\circ R_1\circ \alpha(\psi^\ell_1\circ\psi^\ell_1(t))\overset{\eqref{eq:group-property1}}{=}
\alpha(t)=g(t).
\end{align*}
Notice that we have also used the continuity of $g$ 
established in Lemma~\ref{lem:glueing} to treat the parameter $t=0$, since $\psi_1^\ell(0)=\ell/2$, so that
\[ g(\psi_1^\ell(0))=g(\ell/2)=\lim_{t\nearrow\ell/2}
g(t)\stackrel{\eqref{eq:glueing}}=\lim_{t\nearrow\ell/2}R_1\circ \alpha(\psi_1^\ell(t))=R_1\circ \alpha(0). \]

For $t\in [\ell/4,\ell/2)$ one has $\psi_1^\ell(t)\in (0,\ell/4]$, so that
\[
\tau^\ell_{d_1}(g)(t)\overset{\eqref{eq:group-action}}{=}
R_1\circ g(\psi^\ell_1(t))\overset{\eqref{eq:glueing}}{=}
R_1\circ \alpha(\psi^\ell_1(t))\overset{\eqref{eq:glueing}}{=}g(t),
\]
where again, we have used the continuity of $g$ to treat $t=\ell/4$ 
by means of $g(\psi^\ell_1(\ell/4))=g(\ell/4)=\lim_{t\nearrow \ell/4}g(t) $
similarly as
above.

For $t\in [\ell/2,3\ell/4)$ we have $\psi_{1}^\ell(t)\in (-\ell/4,0]\equiv (3\ell/4,\ell]
\pmod \ell$, so that
\[
\tau^{\ell}_{d_1}(g)(t)\overset{\eqref{eq:group-action}}{=}
R_1\circ g(\psi^\ell_1(t))\overset{\eqref{eq:glueing}}{=}
R_1\circ R_3\circ \alpha(\psi^\ell_3\circ\psi^\ell_1(t))
\overset{\eqref{eq:group-property2}}{=}
R_2\circ \alpha(\psi^\ell_2(t))\overset{\eqref{eq:glueing}}{=}
g(t),
\]
where the parameter $t=\ell/2$ was treated as before via
$g(\psi^\ell_1(\ell/2))=g(0)=\lim_{t\nearrow\ell}g(t)$.
Finally, for $t\in [3\ell/4,\ell)$, one has $\psi^\ell_1(t)\in (-\ell/2,-\ell/4]\equiv (\ell/2,3\ell/4]\pmod\ell$, so that
\[
\tau^\ell_{d_1}(g)(t)\overset{\eqref{eq:group-action}}{=}
R_1\circ g(\psi^\ell_1(t))\overset{\eqref{eq:glueing}}{=}
R_1\circ R_2\circ \alpha (\psi^\ell_2\circ\psi^\ell_1(t))
\overset{\eqref{eq:group-property2}}{=}
R_3\circ \alpha(\psi^\ell_3(t))
\overset{\eqref{eq:glueing}}{=}g(t),
\]
the parameter $t=3\ell/4$ treated by continuity of $g$ via
$g(\psi^\ell_1(3\ell/4))=g(-\ell/4)=g(3\ell/4)=\lim_{t\nearrow 3\ell/4}g(t)$
by $\ell$-periodicity of $g$. 

Thus, we have finished the detailed argument for $i=1$. The case $i=3$ is
analogous, and $i=2$ is somewhat simpler, since there is no inversion in the
domain which saves us the additional continuity argument to treat the respective
boundary parameters.
\end{proof}

Now we can present the

\begin{proof}[Proof of Proposition \ref{prop:character-symmetry}.]
For part (i) assume that $\g\in\Sigma^\ell$ so  $\g=R_i\circ
\g(\psi_i^\ell(\cdot))$ for $i=0,1,2,3$. {Let} $\alpha:=\g|_{[0,\ell/4]}$
and define $g$ according to~\eqref{eq:glueing}. Then
\[
g(t)=R_j\circ\alpha(\psi_j^\ell(t))=R_j\circ\g(\psi_j^\ell(t))=\g(t)\Fo t\in [j\ell/4,(j+1)\ell/4),
\,j=0,1,2,3.
\]
The other implication in part (i) and also parts (ii) and (iii)
follow from Lemmata~\ref{lem:glueing} and \ref{lem:D2-symmetric-curves},  and Corollary \ref{cor:glueing}. Finally, part (iv) follows from the characterizations of
$\Sigma^\ell$, $\Sigma^\ell\cap C^1$ and $\Sigma^\ell\cap W^{k,p}$ for $k=1,2$
in terms of the generating arc $\alpha$, established
in the previous parts (i)--(iii). These characterizations
can be combined with the fact that the  conditions 
\eqref{eq:point-constraints} and     \eqref{eq:tangent-constraints}
on  $\alpha$ are invariant under the reflections $S_i$ for $i=1,2,3$.
\end{proof}

We  use the 
glueing mechanism \eqref{eq:glueing} to construct a few explicit examples of dihedrally
symmetric curves parametrized on $\R/\ell\Z$.
\begin{example}\label{ex:one-circle}
\upshape
As a first generating arc $\alpha_1\in C^\infty([0,\ell/4],\R^3)$ we choose
a quadrant {(i.e., {one} quarter of a circle)}
in the $\mathbf{e_1}$-$\mathbf{e_3}$-plane with arclength $\ell/4$, that is,
\begin{equation}\label{eq:generating-quartercircle}
\alpha_1(t):=\frac\ell{2\pi}\begin{pmatrix}
\sin(2\pi t/\ell) \\
0\\
\cos(2\pi t/\ell)\end{pmatrix}\quad\Fo t\in [0,\ell/4],
\end{equation}
so that $\alpha:=\alpha_1$ 
satisfies the conditions \eqref{eq:point-constraints},
\eqref{eq:tangent-constraints}, and the regularity assumptions of Lemma 
\ref{lem:glueing}, Corollary~\ref{cor:glueing}, and 
Lemma~\ref{lem:D2-symmetric-curves}. According to these results the 
curve $g\equiv g_1$ defined in \eqref{eq:glueing} 
for this particular choice of
$\alpha=\alpha_1$ is a $C^{1,1}$-closed and $D_2$-symmetric curve, that is,
$g_1\in\Sigma^\ell\cap C^{1,1}(\R/\ell\Z,\R^3)$.
It is easy
to check that $g_1$ is the once covered circle whose parametrization equals
\eqref{eq:generating-quartercircle} if one extends the domain of the latter
to all of $[0,\ell]$, so $g_1$ is actually {$C^{\infty}$} on $\R/\ell\Z$; see 
Figure~\ref{fig:tpc_pi}~{(A)}.
\end{example}

\begin{figure}[!t]\tiny
\begin{tabular}{cccc}
\includegraphics[scale=.45,trim=90 110 60 70,clip]{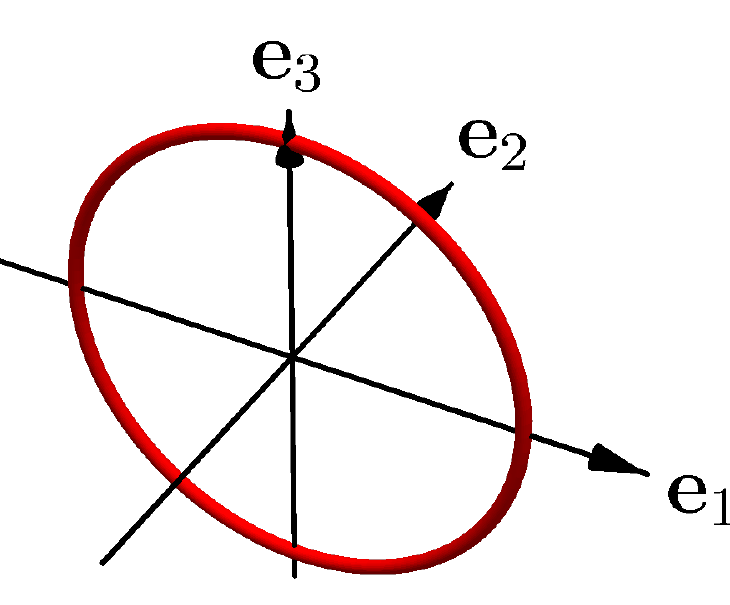}&
\includegraphics[scale=.45,trim=90 110 30 70,clip]{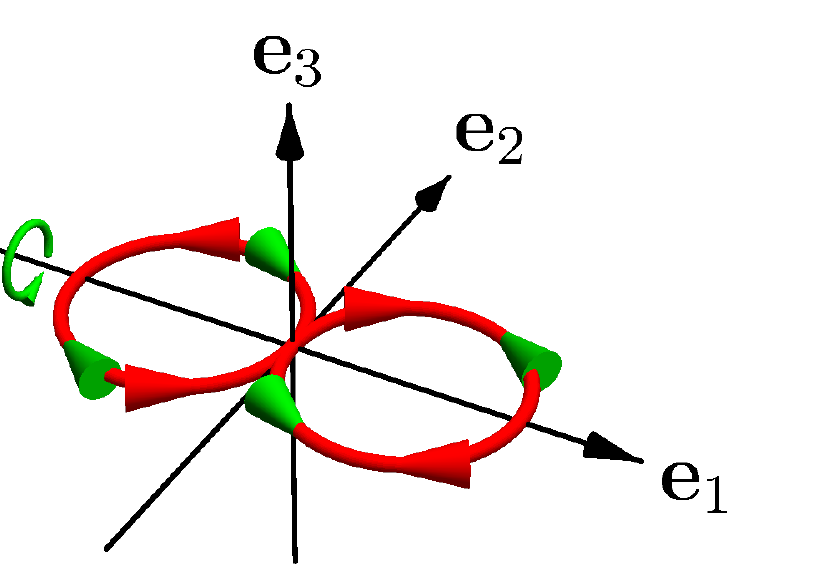}&
\includegraphics[scale=.45,trim=90 110 60 70,clip]{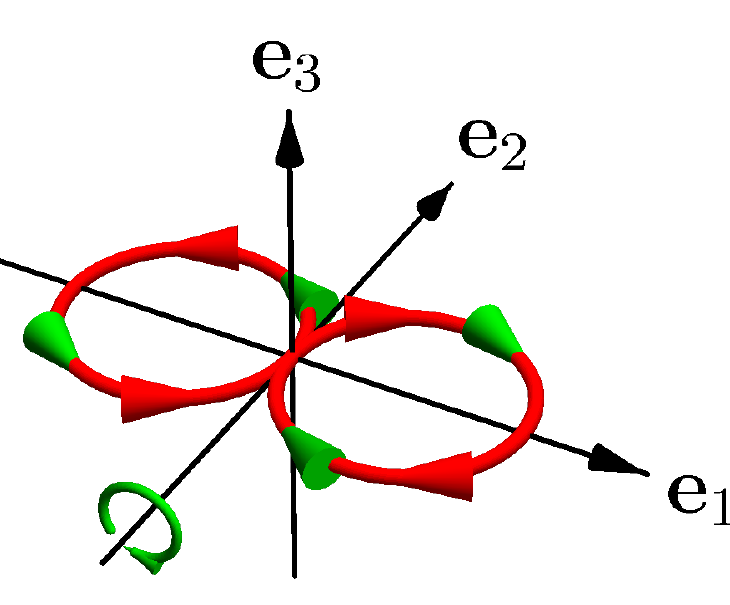}&
\includegraphics[scale=.45,trim=90 110 60 50,clip]{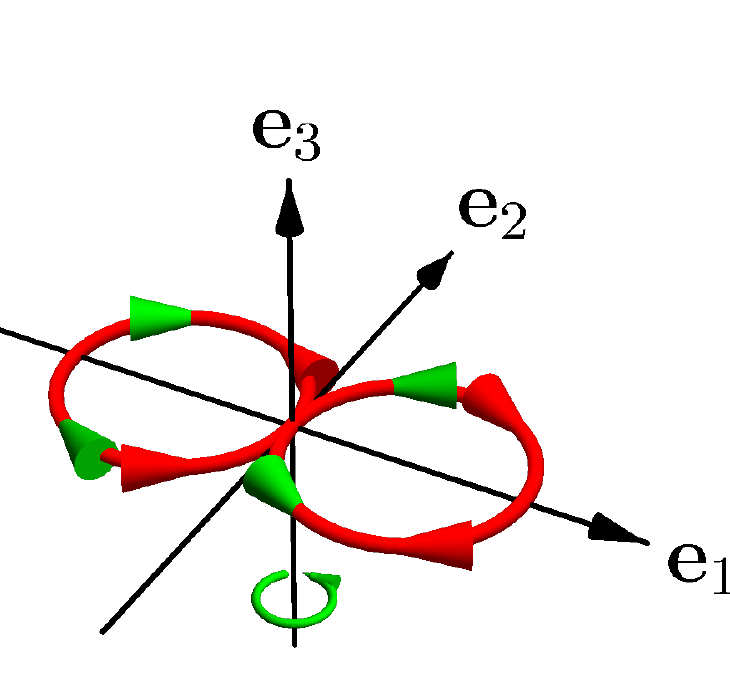}
\\(A)&(B)&(C)&(D)
\end{tabular}
\caption{\label{fig:tpc_pi}
$D_2$-symmetric curves generated by~\eqref{eq:glueing}.
(A) The once covered circle {in the $\mathbf{e_1}$-$\mathbf{e_3}$-plane }
from Example~\ref{ex:one-circle}.
(B)--(D) A tangential pair of co-planar circles
exhibiting symmetries with respect to the coordinate axes
as described in Example~\ref{ex:tpc_pi}.
{$180$-degree rotations about the coordinate axes $\R\mathbf{e_{\boldsymbol{i}}}$ 
indicated in green are accompanied
by parameter transformations $\psi_i^\ell$ that reverse orientation for $i=1,3$ in 
(B) and (D), while
preserving it for $i=2$ {in}~(C); see red and green arrows along the curves.}
}
\end{figure}

\begin{example}\label{ex:tpc_pi}
\upshape
We construct a dihedrally symmetric
tangential pair of co-planar circles with exactly one self-intersection point
within the $\mathbf{e_1}$-$\mathbf{e_2}$-plane; see  Figure~\ref{fig:tpc_pi}~{(B)--(D)}.
For that we take as a generating arc $\alpha_2\in C^\infty([0,\ell/4],\R^3)$
a semicircle of radius $\ell/(4\pi)$. To be more precise, we set
\begin{equation}\label{eq:generating-semicircle}
\alpha_2(t):=\frac\ell{4\pi}
\begin{pmatrix}
1-\cos(4\pi t/\ell)\\
\sin(4\pi t/\ell)\\
0\end{pmatrix}\quad\Fo t\in [0,\ell/4],
\end{equation}
and easily check that conditions \eqref{eq:point-constraints} 
and \eqref{eq:tangent-constraints} 
as well as the regularity assumptions of Lemma~\ref{lem:glueing},
Corollary~\ref{cor:glueing}, and Lemma~\ref{lem:D2-symmetric-curves} are satisfied
for $\alpha:=\alpha_2$.
Glueing according to \eqref{eq:glueing} yields the $D_2$-symmetric
 tangential pair of co-planar
circles with parametrization $g\equiv g_2\in\Sigma^\ell\cap C^{1,1}(\R/\ell\Z,\R^3)$,
which in the case $\ell=1$
we also denote by $\tpc[\pi]$. This curve has the same trace
as the corresponding curve in 
\cite[Formula (3.2) for $\varphi:=\pi$]{gerlach-etal_2017}, only with reversed
orientation.

In a similar manner, we may construct a $\tpc[\pi]$-curve in the $\mathbf{e_{2}}$-$\mathbf{e_{3}}$-plane by letting
\begin{equation}\label{eq:generating-semicircle2}
\alpha_2(t):=\frac\ell{4\pi}
\begin{pmatrix}
0\\
\sin(4\pi t/\ell)\\
1+\cos(4\pi t/\ell)
\end{pmatrix}\quad\Fo t\in [0,\ell/4].
\end{equation}
{We will show in Corollary~\ref{cor:circle}
that \eqref{eq:generating-semicircle}
and \eqref{eq:generating-semicircle2} are the \emph{only two options}
for the generating arc to construct $D_2$-symmetric tangential pairs of circles
$\tpc[\pi]$.}
\end{example}

Our third example produces dihedrally symmetric torus knots 
of class $\mathcal{T}(2,b)$
for any odd $b\in\Z\setminus\{1,-1\}$; see Figure~\ref{fig:constrB}.

\begin{example}\label{ex:D2-symmetric-torus-knot}\upshape
The generating curve $\alpha_3$ consists of a helical part $h$, 
which after the glueing
forms together with its rotated copies a rational tangle 
that determines the knot
class \cite[Section 2.3]{adams_2004}, 
and a piece $\sigma$ of a stadium curve, 
which after the glueing
closes the tangle to form the knot; see Figure~\ref{fig:constrB}.
\begin{figure}
 \includegraphics[scale=.3]{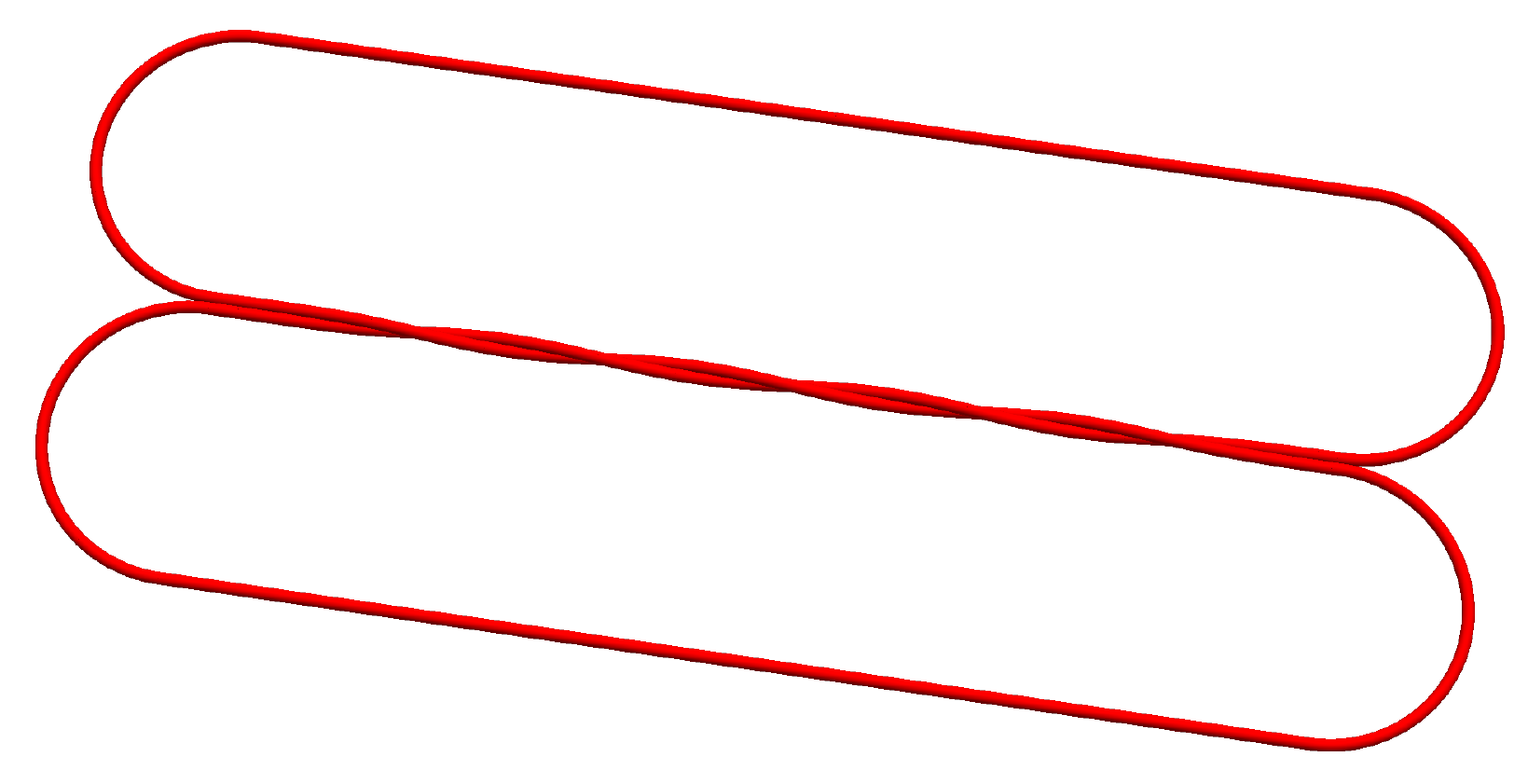}
  \includegraphics[scale=.4]{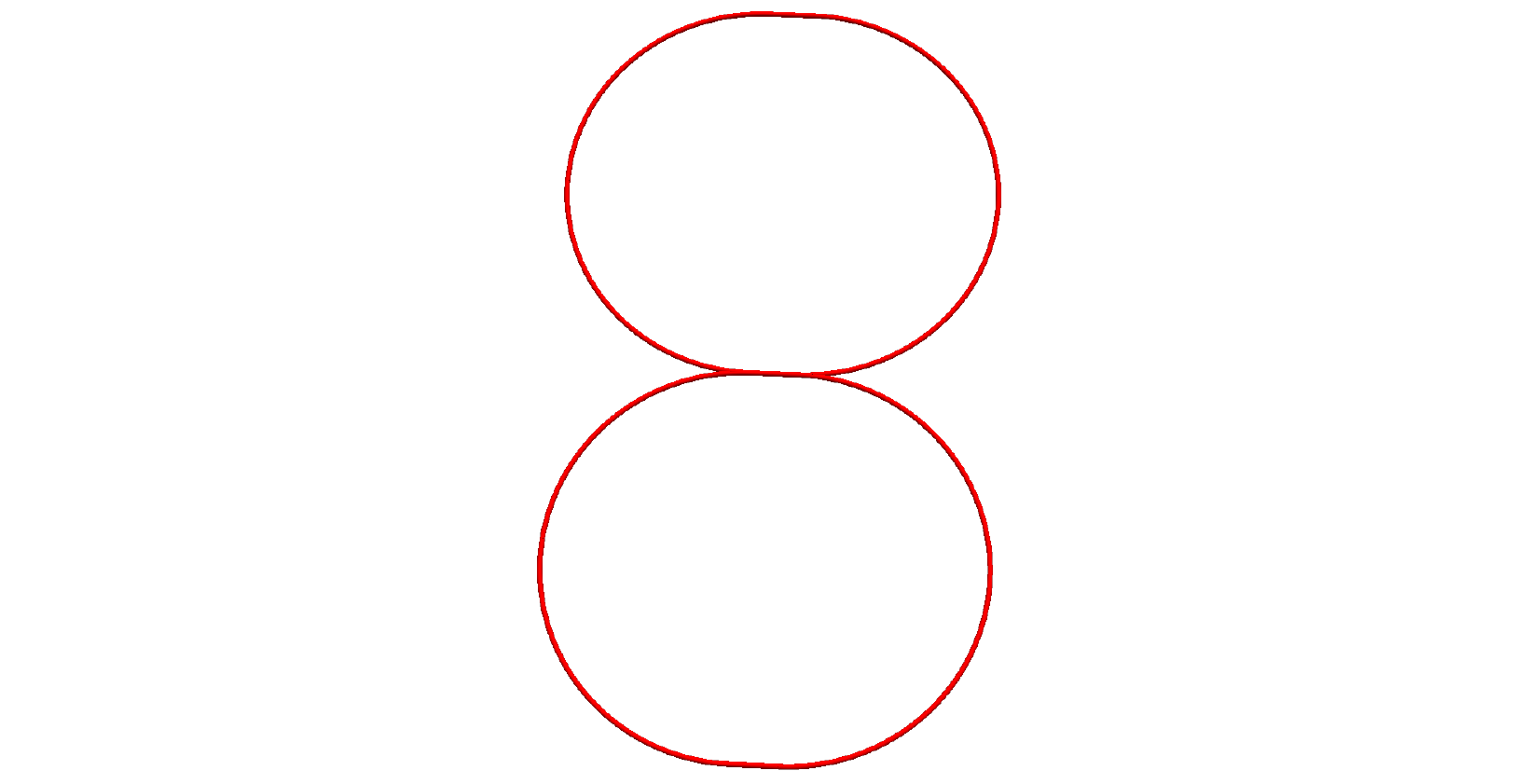}
 \caption{Dihedrally symmetric torus knots of class
$\TL(2,5)$ constructed in Example~\ref{ex:D2-symmetric-torus-knot}
 for $\epsilon=0.03$ (left) and for $\epsilon=0.003$
 (right).}\label{fig:constrB}
  \end{figure}
   For the precise formulas,
which we are going to take up again in Section~\ref{sec:example}
to compute the infimal
bending energy on $D_2$-symmetric torus
knots,
{it suffices to consider a fixed odd integer $b\ge 3$.
Indeed, the reflection of a $D_2$-symmetric
representative of the torus knot class
$\mathcal{T}(2,b)$ in a coordinate plane $\mathbf{e_{\boldsymbol{i}}}$ for $i=1,2,3$,
produces a representative of $\mathcal{T}(2,-b)$ that still
possesses the $D_2$-symmetry
 according to part (iv) of
Proposition
\ref{prop:character-symmetry}.  We also fix}
two parameters $\epsilon,\rho>0$,
and define 
the helical part $h=h^\epsilon$ {of the  generating arc} as 
\begin{equation}\label{eq:helical-part}
h^\epsilon(t):= \begin{pmatrix}
\rho(-1)^{(b-1)/2}\sin\phi_\epsilon(t)\\
t\\
\rho\cos\phi_\epsilon(t)
\end{pmatrix}\quad\Fo t\in [0,\infty),
\end{equation}
where $\phi_\epsilon(t):=\pi\cdot\phi(t/\epsilon)$ for the piecewise smooth
parameter transformation $\phi\in
C^{1,1}([0,\infty))$ {given by}
\begin{equation}\label{eq:phi}
\phi(t):=\begin{cases}
t & \Fo t\in\left[0,\frac{b-1}2\right],\\
\frac{b}2 -\frac12 \left(t-\frac{b+1}2\right)^2 & \Fo t\in \left[\frac{b-1}2,
\frac{b+1}2\right],\\
\frac{b}2 & \Fo t\ge\frac{b+1}2.
\end{cases}
\end{equation}
Note that
\begin{equation}\label{eq:h^eps}
h^{\eps}(\tfrac{(b+1)\eps}2)
=\begin{pmatrix}
\rho(-1)^{(b-1)/2}\sin\frac{b\pi}2 \\
\tfrac{(b+1)\eps}2\\
0
\end{pmatrix}
=\begin{pmatrix}
\rho \\
\tfrac{(b+1)\eps}2\\
0
\end{pmatrix}.
\end{equation}
The portion $\sigma=\sigma^\epsilon$ of a stadium curve of class $C^{1,1}$
consists of a semicircle of
radius $r=r(\epsilon)$ satisfying 
\begin{equation}\label{eq:r-condition}
(b+1)\epsilon+\pi r=\ell/4,
\end{equation}
and
a (short) straight segment attached in a $C^1$-manner to the 
semicircle. The precise definition is
\begin{equation}\label{eq:stadium}
\sigma^\epsilon(t):=\begin{cases}
\begin{pmatrix}
\rho +r -r\cos\big(\frac1r (t-(b+1)\epsilon/2)\big)\\
(b+1)\epsilon/2+r\sin\big(\frac1r (t-(b+1)\epsilon/2)\big)\\
0
\end{pmatrix} & \Fo t\in \left[\frac{(b+1)\epsilon}2,\frac{\ell}4 -
\frac{(b+1)\epsilon}2\right],\\ & \\
\begin{pmatrix}
\rho +2r\\
(\ell/4)-t\\
0
\end{pmatrix} & \Fo t\in \left[\frac{\ell}4 -\frac{(b+1)\epsilon}2,
\frac{\ell}4\right].
\end{cases}
\end{equation}
The generating arc $\alpha^\epsilon_3$ for 
fixed $\epsilon>0$ {satisfying~\eqref{eq:r-condition}} is now defined as
\begin{equation}\label{eq:generating-torus-knot}
\alpha^\epsilon_3(t):=\begin{cases}
h^\epsilon(t) & \Fo t\in\left[0,\frac{b+1}2\cdot\epsilon\right], \\
\sigma^\epsilon(t) & \Fo t\in \left(\frac{b+1}2\cdot\epsilon,\frac{\ell}4\right],
\end{cases}
\end{equation}
and one easily checks with the help of~\eqref{eq:r-condition}
and~\eqref{eq:h^eps}
that $\alpha_3^\epsilon$ itself is of class
$C^{1,1}$ on $[0,\ell/4]$. Moreover,
$\alpha_3^\epsilon(0)=h^\epsilon(0)=(0,0,\rho)^T\in\R\mathbf{e_3}$
and $\alpha_3^\epsilon(\ell/4)=(\rho + 2r,0,0)^T\in\R\mathbf{e_1}$
as required in \eqref{eq:point-constraints} of Lemma~\ref{lem:glueing}.
Finally, $(\alpha_3^\epsilon)'(0)=(\rho(-1)^{(b-1)/2}\pi/\epsilon,1,0)^T\in
\spann\{\mathbf{e_1},\mathbf{e_2}\}$ and $(\alpha_3^\epsilon)'(\ell/4)=
(0,-1,0)^T\in\spann\{\mathbf{e_2},\mathbf{e_3}\}$, so that also
\eqref{eq:tangent-constraints} is satisfied. Therefore, all assumptions
of Lemma~\ref{lem:glueing}, Corollary~\ref{cor:glueing}, and
Lemma~\ref{lem:D2-symmetric-curves} are satisfied, so that the
glueing \eqref{eq:glueing} yields a $D_2$-symmetric torus knot $g^\epsilon_3
\in \Sigma^\ell\cap C^{1,1}(\R/\ell\Z,\R^3)$ representing the
knot class $\mathcal{T}(2,b)$.

We establish in Section~\ref{sec:example}
the strong $W^{2,2}$-convergence of these torus
knots $g^\epsilon_3$ for $\ell=1$ to the tangential pair of 
co-planar circles constructed in Example~\ref{ex:tpc_pi}, i.e., to  $\tpc[\pi]$ as $\epsilon\to 0.$
\end{example}

In light of the above examples, it is in order
to ask about the location of the dihedrally symmetric curves
appearing in the main theorems, namely the round circle
and the tangential pair of circles $\tpc[\pi]$.

\begin{corollary}[$D_{2}$-symmetric circle and ${\tpc[\pi]}$]\label{cor:circle}
 Up to reparametrization there are a unique $D_{2}$-symmetric circle and precisely two $D_{2}$-symmetric $\tpc[\pi]$-curves
 (which are of course isometric).
 \begin{enumerate}
 \item Any $D_{2}$-symmetric circle $c:\R/\Z\to\R^{3}$ is centered at the origin
 and contained in the plane perpendicular to~$\mathbf{e_{2}}$
 with initial point $c(0)\in\R\mathbf{e_3}$.
 \item A $D_{2}$-symmetric $\tpc[\pi]:\R/\Z\to\R^{3}$
 is located in {one of the two coordinate planes $\mathbf{e_1^\perp}$
 and $\mathbf{e_3^\perp}$.}
 The self-intersection point is at the origin and its
 tangent is parallel to $\mathbf{e_{2}}$ in both cases.
 \end{enumerate}
\end{corollary}

\begin{proof}
(i)~The center $M$ of any circle $C\subset\R^3$, 
which -- as a set -- is $D_2$-symmetric, is the origin, 
since if not, we could find $i\in\{1,2,3\}$, such that $R_iM\not=M$ and
hence $R_iC$ is a circle of the same length as $C$, but
with a center $R_iM$ different from $M$.
Therefore, {the sets} $R_iC$ {and}
$C$ {differ which contradicts} the dihedral symmetry of~$C$.
Consequently, any injective 
arclength parametrization $c:\R/\Z\to\R^3$ of the once covered circle
with $D_2$-symmetry satisfies the antipodal relation
$c(t)=-c(t+1/2)$ for all $t\in\R/\Z$. 
If there is one such parametrization $c$ with $D_2$-symmetry
then its image
$c(\R/\Z)$ must be contained in $\spann\{\mathbf{e_1},\mathbf{e_3}\}$
since, 
by $1$-periodicity,
\begin{align}\label{eq:circle-symmetry}
c(t)& 
\stackrel{{\eqref{eq:group-action}}}=
R_2\circ c(t-1/2)
 = R_2\circ c(t+1/2)=-R_2\circ c(t)%
\end{align}
{for all $t\in\R/\Z$}.
Moreover, 
{we infer}
$c(0)\in\R\mathbf{e_3}$
from {\eqref{eq:point-constraints} 
in Proposition~\ref{prop:character-symmetry}.} 
The existence of such {a} parametrization
{was established in} Example~\ref{ex:one-circle}, 
cf.~\eqref{eq:generating-quartercircle}.

(ii)~For $\tpc[\pi]$ we can argue in a similar way
to see that the self-intersection point is at the origin.
{By Lemma~\ref{lem:arclength-sym} below we may assume that
$\tpc[\pi]$ is an arclength parametrized curve $\R/\Z\to\R^{3}$.}
Let the unit tangent {of $\tpc[\pi]$}
at {the origin} be denoted by $\nu\in\S^{2}$.
Due to symmetry,
$\tpc[\pi](t+\tfrac12)$ is just the image of {a $180$-degree rotation $R_\nu$
of $\tpc[\pi](t)$ about~$\nu$.}
Thus we obtain
\[ R_{\nu}\circ\tpc[\pi](t) = \tpc[\pi](t+\tfrac12)
\stackrel{{\eqref{eq:group-action}}}= R_{2}\circ\tpc[\pi](t)
 \qquad\text{for all } t\in\R/\Z. \]
This implies that the matrix product $R_{\nu}^{-1} R_{2}$
is the identity on the hyperplane which contains (the image of)
$\tpc[\pi]$. As it belongs to $SO(3)$, it must even be $\Id_{\R^{3}}$, 
in particular $\nu=\mathbf{e_{2}}$.

From part (i) of Proposition \ref{prop:character-symmetry}
we infer
that the image of $\tpc[\pi]$ must contain points
in $\R\mathbf{e_{\boldsymbol{k}}}\setminus\set0$ for $k=1$ or $k=3$.
Together with $\nu=\mathbf{e_{2}}$ we conclude that
the planar curve
$\tpc[\pi]$ belongs either to {$\mathbf{e_{1}^{\perp}}$
or $\mathbf{e_{3}^{\perp}}$}.
Both configurations are realized by Example~\ref{ex:tpc_pi},
cf.~\eqref{eq:generating-semicircle} and~\eqref{eq:generating-semicircle2}.
\end{proof}

According to the following result one can reparametrize $D_2$-symmetric
curves without affecting the symmetry. This, as well as the subsequent
uniform a priori bound on the size of dihedrally symmetric curves, turns out
to be useful ingredients in the existence proofs of Section~\ref{sec:existence}.

\begin{lemma}[Symmetry of arclength parametrization]\label{lem:arclength-sym}
If  $\g\in\Sigma^\ell$ has length $\mathscr{L}(\g)=L$, then
its arclength
parametrization $\Gamma$ is also $D_2$-symmetric, that is, $\Gamma$
is contained in the $D_2$-symmetric set $\Sigma^L$ defined as in 
\eqref{eq:Sigma}.
\end{lemma}
\begin{proof}
Differentiating the symmetry relation $\g=\tau^\ell_{d_i}(\g)$ on $\R/\ell\Z$
we obtain by \eqref{eq:group-action} and \eqref{eq:inner-action}
\begin{equation}\label{eq:speed}
\abs{\g'(t)}=\abs{\tau^\ell_{d_i}(\g)'(t)}\overset{\eqref{eq:group-action}}{=}
\abs{R_i\circ\g'\big(\psi^\ell_i(t)\big)(\psi^\ell_i)'(t)}
\overset{\eqref{eq:inner-action}}{=}
\abs{\g'\big(\psi^\ell_i(t)\big)}
\end{equation}
for all $t\in\R/\ell\Z$ and $i=0,1,2,3.$ This identity can be used to compute
the arclength parameter
\begin{align}\label{eq:arclength-computation}
s\big(\psi^\ell_i(t)\big) & =
{\length\br{\gamma|_{[0,\psi^\ell_i(t)]}}}
=\int_0^{\psi^\ell_i(t)}\abs{\g'(\tau)}\,d\tau
\overset{\eqref{eq:speed}}{=}
\int_0^{\psi^\ell_i(t)}\abs{\g'(\psi_i^\ell(\tau))}\,d\tau\notag\\
& = \int_{\psi^\ell_i(0)}^t\abs{\g'(z)}\frac1{(\psi^\ell_i)'({\psi^\ell_i}(z))}\,dz=
\br{\sign (\psi_i^\ell)'}\big[s(t)-s(\psi_i^\ell(0))\big],
\end{align}
where we changed variables to $z:=\psi_i^\ell(\tau)$ with $z(0)=\psi_i^\ell(0)$,
$z(\psi_i^\ell(t))=\psi_i^{\ell}\circ\psi_i^\ell(t)=t$ by virtue
of \eqref{eq:group-property1}.
Notice also that $(\psi_i^\ell)'(\cdot)=(-1)^i$ for $i=0,1,2,3$.
{Using~\eqref{eq:arclength-computation} with $i=2$ and $t=\ell$ we infer
$2s(\frac\ell2)=s(\ell)=L$. Now} 
it is easy to check that
\begin{equation}\label{eq:psi-bridge}
\br{\sign (\psi_i^\ell)'}\big[s(t)-s(\psi_i^\ell(0))\big]=\psi_i^L(s(t))\quad\Foa
t\in\R/\ell\Z,\,i=0,1,2,3,
\end{equation}
where $\psi_i^L$ is the transformation defined in \eqref{eq:inner-action} only
with $\ell$ replaced by $L$. In other words, $D_2$ acts on the domain $\R/L\Z$
of the arclength parametrization $\Gamma$ via the transformations $\psi_i^L$, $i=
0,1,2,3.$ Combining \eqref{eq:arclength-computation} with \eqref{eq:psi-bridge}
we arrive at
\begin{equation}\label{eq:arclength-period}
s\big(\psi^\ell_i(t)\big)=\psi_i^L(s(t))\quad\Foa
t\in\R/\ell\Z,\,i=0,1,2,3,
\end{equation}
so that the symmetry of $\g$ leads to 
\begin{align}\label{eq:arclength-sym2}
\Gamma\big(s(t)\big) & =\g(t) = \tau^\ell_{d_i}(\g)(t)\overset{\eqref{eq:group-action}}{=}
R_i\circ\g(\psi_i^\ell(t))=R_i\circ\Gamma\big(s(\psi_i^\ell(t))\big)\notag\\
& \overset{\eqref{eq:arclength-period}}{=}R_i\circ\Gamma\big(\psi_i^L(s(t))\big)
\overset{\eqref{eq:group-action}}{=}\tau^L_{d_i}(\Gamma)\big(s(t)\big)
\end{align}
for all $t\in\R/\ell\Z,\,i=0,1,2,3,$ which establishes the symmetry of $\Gamma$.
\end{proof}

\begin{lemma}[Optimal $L^\infty$-bound]\label{lem:improved-Linfty-bound}
A closed curve $\g\in C^0(\R/\ell\Z,\R^3)$  
of length
$\Lambda\in (0,\infty)$ whose image has dihedral symmetry is contained
in the closure of the ball $B_{\Lambda/4}(0)$. 
\end{lemma}

\begin{proof}
Assume to the contrary that there is a point 
$(x_1,x_2,x_3)=x:=\gamma(s)$ such that
$|x|>\Lambda/4.$ We may assume without loss of generality that $|x_1|\ge |x_2|\ge
|x_3|$. The symmetry assumption means 
\begin{equation}\label{eq:symmetry-assumption}
R_1\circ\gamma(\R/\ell\Z)=
R_2\circ\gamma(\R/\ell\Z)=
R_3\circ\gamma(\R/\ell\Z)=\gamma(\R/\ell\Z),
\end{equation}
so that 
we have $R_i(x)\in\gamma(\R/\ell\Z)$ for $i=1,2,3$, and some permutation
of the four points $A:=x$, $B:=R_1(x)$, $C:=R_2(x)$, $D:=R_3(x)$
forms a polygon inscribed in $\gamma$. By direct computation
we infer
\begin{align}\label{eq:edge-comparison}
S:=\abs{A-B}=\abs{C-D}= 2\sqrt{x_2^2+x_3^2},\notag\\
M:=\abs{A-C}=\abs{B-D}=2\sqrt{x_1^2+x_3^2},\\
L:=\abs{A-D}=\abs{B-C}=2\sqrt{x_1^2+x_2^2},\notag
\end{align}
with $0\le S\le M\le L$ according to our assumption on the coordinates of $x$.
Of all the possible choices of permutations of the points $A,B,C,D$, the two
closed polygons $P_1:=ABDCA$ and $P_2:=ACDBA$ have the shortest length
$\mathscr{L}(P_1)=\mathscr{L}(P_2)=2S+2M$, which by means of 
\eqref{eq:edge-comparison} leads to the contradictive
inequality
\[ \Lambda\ge 2S+2M\overset{\eqref{eq:edge-comparison}}{=}
4\big(\sqrt{x_2^2+x_3^2}+\sqrt{x_1^2+x_3^2}\big)\ge
4|x|>\Lambda. \qedhere \]
\end{proof}

This $L^\infty$ bound is optimal, since one can think of 
a sequence of ellipses of length $L$, all centered at the origin and
 contained in a fixed coordinate plane, converging
 to a straight segment of length $L/2$ on one coordinate
axis. All such ellipses are contained in the ball of radius $L/4$
centered at the origin (cf.  \cite{nitsche_1971}).

Recall from the introduction the set 
$W^{2,2}_\textnormal{ir}(\R/\ell\Z,\R^3)$
of closed, regular and embedded $W^{2,2}$-curves, each of which represents
a tame knot class.
Thus, for a given  knot class $\mathcal{K}$ we introduce
the subset
\begin{equation}\label{eq:Omega_K}
\Omega^\ell_\mathcal{K}:=\{\g\in W^{2,2}_\textnormal{ir}(\R/\ell\Z,\R^3):
[\g]=\mathcal{K}\},
\end{equation}
{which -- according to the Morrey--Sobolev embedding $W^{2,2}\hookrightarrow
C^1$ -- is the empty set unless $\mathcal K$ is
tame; see Footnote \ref{foot:tame}.
First we observe
 that $\Omega^\ell_\mathcal{K}$ is a Banach manifold.}

\begin{lemma}[$\Omega^\ell_\mathcal{K}$ is a Banach manifold]
\label{lem:Omega-banach}
For any fixed tame knot class $\mathcal{K}$  the set $\Omega^\ell_\mathcal{K}$
 is a {non-empty} open subset of 
the Banach space $W^{2,2}(\R/\ell\Z,\R^3)$, hence a Banach manifold.
\end{lemma}
\begin{proof}
By the Morrey--Sobolev embedding result any curve 
$\g\in\Omega^\ell_\mathcal{K}$
is a regular $C^1$-knot representing the knot class $\mathcal{K}$, so that according
to 
Lemma~\ref{lem:isot} the curve $\g$ possesses
a neighbourhood $\mathcal{U}\subset C^1(\R/\ell\Z,\R^3)$ such that any curve 
$\xi\in\mathcal{U}$ is regular and of the same knot type $\mathcal{K}$. Again by means of
the Morrey--Sobolev embedding theorem we can choose the
radius $\delta$ of the 
 ball $B_\delta(\g)\subset W^{2,2}(\R/\ell\Z,\R^3)$ so small that  $B_\delta(\g)
 \subset\mathcal{U}$, which proves the claim.
 \end{proof}
Restricting the group action \eqref{eq:group-action} to $\Omega^\ell_\mathcal{K}$ yields a smooth $D_2$-manifold.
\begin{lemma}[$\Omega^\ell_\mathcal{K}$ is $D_2$-manifold]
\label{lem:group-action-Omega}
The mapping $\tau^\ell$ defined in \eqref{eq:group-action} acts on
the Banach manifold $\Omega^\ell_\mathcal{K}$, and under this action
$\Omega^\ell_\mathcal{K}$ becomes a smooth $D_2$-manifold.
\end{lemma}
\begin{proof}
It is easy to see that $\tau^\ell_{d_i}(\g)$ is contained in $\Omega^\ell_\mathcal{K}$
for $i=0,1,2,3$, and $\g\in\Omega^\ell_\mathcal{K}$,
since any rotation in the image and affine linear transformation of the periodic domain
does not change the $W^{2,2}$-regularity and injectivity on $[0,\ell)$. Moreover,
the knot class $\mathcal{K}$ is preserved as well, and
\[
\abs{\tau^{\ell}_{d_i}(\g)'(t)}=\abs{R_i\circ\g'(\psi_i^{\ell}(t))\br{\psi_i^{\ell}}'(t)}
=\abs{\g'(\psi_i^{\ell}(t))}>0\Foa t\in\R/\ell\Z.
\]
The algebraic property \eqref{eq:homo} as well as the
 {linearity} of
$\tau^{\ell}_{d_i}:\Omega^\ell_\mathcal{K}\to\Omega^\ell_\mathcal{K}$ for $i=0,1,2,3$
was
verified in the proof of Lemma~\ref{lem:group-action-C0}. Indeed, the linearity of
$\tau^{\ell}_{d_i}$ leads to the differential
\[
(d\tau^{\ell}_{d_i})_\g\eta=\tau^{\ell}_{d_i}(\eta)\quad\Foa \g\in\Omega^\ell_\mathcal{K},\,\,
\eta\in T_\g\Omega^\ell_\mathcal{K}\simeq W^{2,2}(\R/\ell\Z,\R^3)
\]
{and} $i=0,1,2,3.$
Therefore, $\Omega^\ell_\mathcal{K}$ is a smooth $D_2$-manifold, since $
\tau^{\ell}_{d_i}:\Omega^\ell_\mathcal{K}\to\Omega^\ell_\mathcal{K}$ is a
diffeomorphism with (smooth) inverse $\br{\tau^{\ell}_{d_i}}^{-1}:=\tau^{\ell}_{d_i}$
for each $i=0,1,2,3$ by means or the properties \eqref{eq:group-property1}.
\end{proof}

\section{Existence theory under the $D_2$-symmetry constraint}\label{sec:existence}
Throughout this section we set $\ell=1$.
Instead of the total energy $E_\vartheta=E+\vartheta\TP_q^{1/(q-2)}$ which is
positively $(-1)$-homogeneous, i.e., $E_\vartheta(r\g)=r^{-1}E_\vartheta(\g)$ for
all $\g\in W^{2,2}(\R/\Z,\R^3)$ and $r>0$, we consider the scale-invariant
version
\begin{equation}\label{eq:S_theta}
S_\vartheta(\g):=\mathscr{L}(\g)\cdot E_\vartheta(\g)=\mathscr{L}(\g)\cdot
\left( E(\g)+\vartheta \TP_q^{\frac1{q-2}}(\g)\right).
\end{equation}
We first minimize this scale-invariant total energy  on the class
of $W^{2,2}$-knots with dihedral symmetry, that is, we minimize 
$S_\vth$ on the $D_2$-symmetric subset 
\begin{equation}\label{eq:symmetric-set}
\Sigma_\mathcal{K}:=\Sigma^1\cap \Omega^1_\mathcal{K},
\end{equation}
where $\Sigma^\ell$ for general $\ell>0$ was defined in \eqref{eq:Sigma}
{and $\Omega^\ell_\mathcal{K}$ in~\eqref{eq:Omega_K}}.
\begin{theorem}[Symmetric minimizers of total scaled energy]
\label{thm:existence-total}
Assume that $\Sigma_\mathcal{K}\not=\emptyset $ for a given 
knot class $\mathcal{K}$. Then for any $\vartheta>0$ there exists an 
arclength parametrized knot
$\Gamma_\vartheta\in \Sigma_\mathcal{K}$ with length $\mathscr{L}
(\Gamma_\vth)=1$, such that
\begin{equation}\label{eq:mini-total}
S_\vartheta(\Gamma_\vartheta)=\inf_{\Sigma_\mathcal{K}}S_\vartheta(\cdot).
\end{equation}
\end{theorem}
Before proving this crucial existence result, let us draw some immediate
conclusions that also lead to the proof{s of Theorems~\ref{thm:symmetric-critical}
and  \ref{thm:symmetric-elastic-knots}} stated
in the introduction.
\begin{corollary}[Symmetric $S_\vartheta$-critical points]\label{cor:symm-critic}
Any symmetric  locally minimizing knot $\g\in\Sigma_\mathcal{K}$ 
of $S_\vartheta|_{\Sigma_\mathcal{K}}$ is $S_\vartheta$-critical,
that is,
\begin{equation}\label{eq:symm-critic}
DS_\vartheta(\g) h = 0\quad\Foa h\in W^{2,2}(\R/\Z,\R^3).
\end{equation}
\end{corollary}
\begin{proof}
Small $D_2$-symmetric variations of a locally minimizing knot $\g\in
\Sigma_\mathcal{K}
$ remain regular and in the same knot class $\mathcal{K}$, so that $\g$
is $(S_\vartheta|_{\Sigma_\mathcal{K}})$-critical, that is,
$D(S_\vartheta|_{\Sigma_\mathcal{K}})(\g)h=0$ for all $h\in T_\g
\Sigma_\mathcal{K}$.  Using the definition of the group action
\eqref{eq:group-action}, \eqref{eq:inner-action} one easily checks
that $S_\vartheta$ is a $D_2$-invariant energy.
{Indeed, both $E$ and $\TP_{q}$ are invariant under
Euclidian transformations and reparametrization;
here we even do not change the speed due to
$|\br{\psi_i^{\ell}}'(\cdot)|=1$ for all $i=0,1,2,3$.
Furthermore, $S_\vartheta$}
is of class $C^1$ by means of Theorem~\ref{thm:reg-tanpoint}.

Moreover,
$\Sigma_\mathcal{K}$ is the non-empty $D_2$-symmetric subset of the
smooth $D_2$-manifold $\Omega^1_\mathcal{K}$, {cf.}\ Lemma 
\ref{lem:group-action-Omega}, so that
we can apply the version
of Palais's principle of 
symmetric criticality stated in Corollary~\ref{cor:psk}.
\end{proof}
Criticality of $S_\vartheta$ is directly related to criticality for the constrained 
variational problem \eqref{eq:P_theta} for the original total energy $E_\vartheta$.
\begin{corollary}[Euler--Lagrange-equation]\label{cor:ELG}
Any arclength parametrized critical point $\Gamma\in W^{2,2}(\R/\Z,\R^3)$ of $S_\vartheta$ 
satisfies
\begin{equation}\label{eq:ELG}
DE_\vartheta(\Gamma)+\lambda\cdot D\mathscr{L}(\Gamma)=0
\end{equation}
for the Lagrange multiplier $\lambda:=E_\vartheta(\Gamma)$. 
Moreover, any arclength parametrized critical point 
for the variational problem \eqref{eq:P_theta} satisfies the same variational 
equation \eqref{eq:ELG}.
\end{corollary}
\begin{proof}
The Euler--Lagrange equation \eqref{eq:ELG} is a direct consequence of
\eqref{eq:symm-critic} via the product rule and because $\mathscr{L}(\Gamma)=1$.
 The  variational
equation for the constrained  variational problem \eqref{eq:P_theta} is
\begin{equation}\label{eq:constrained-ELG}
DE_\vartheta(\Gamma)+\mu\cdot D\mathscr{L}(\Gamma)=0
\end{equation}
for some Lagrange multiplier $\mu\in\R,$  since small variations $\Gamma
+\epsilon h$
for $h\in W^{2,2}(\R/\Z,\R^3)$ remain regular and in the same knot
class $\mathcal{K}$ by Morrey's compact embedding $W^{2,2}$ into $C^1$.
Testing \eqref{eq:constrained-ELG} with $\Gamma$ itself
we can use the fact that $\abs{\Gamma'}\equiv 1$ to find
\[
0=DE_\vartheta(\Gamma)\Gamma +\mu \int_{\R/\Z}\abs{\Gamma'(\tau)}^2\,d\tau=
DE_\vartheta(\Gamma)\Gamma+\mu,
\]
so that $\mu=-DE_\vartheta(\Gamma)\Gamma=E_\vartheta(\Gamma)$ 
by the positive
$(-1)$-homogeneity of $E_\vartheta$.
\end{proof}

{Before proving Theorem~\ref{thm:existence-total}
itself, we turn to another immediate application, namely
the existence of symmetric critical knots for the total energy $E_\vth$ as stated in
Theorem~\ref{thm:symmetric-critical}, and the existence of symmetric elastic knots; see Theorem~\ref{thm:symmetric-elastic-knots}.}

\begin{proof}[Proof of Theorem~\ref{thm:symmetric-critical}]
Since by assumption there is at least one $D_2$-symmetric knot contained in $\CK$ one has $\Sigma_\mathcal{K}\not=\emptyset$ so that
Theorem~\ref{thm:existence-total} is applicable.
The $S_\vartheta$-minimizing knots $\Gamma_\vartheta
\in\Sigma_\mathcal{K}$ obtained in that theorem 
have length one, so $\Gamma_\vth\in\CK$. They 
are $S_\vartheta$-critical according to Corollary~\ref{cor:symm-critic}. Moreover,
$\abs{\Gamma_\vartheta'}\equiv 1$ on $\R/\Z$ so that Corollary~\ref{cor:ELG}
implies that the Euler--Lagrange equation \eqref{eq:ELG} holds true, which is
the variational equation \eqref{eq:intro-ELG}
stated in Theorem~\ref{thm:symmetric-critical} with the
exact same Lagrange multiplier. 
\end{proof}

\begin{proof}[Proof of Theorem~\ref{thm:symmetric-elastic-knots}]
For any $\vartheta >0$ we find by virtue of Theorem~\ref{thm:existence-total}
an arclength parametrized knot $\Gamma_\vartheta\in\Sigma_{\mathcal{K}}$
(of length $\mathscr{L}(\Gamma_\vartheta)=1$ hence $\Gamma_\vth\in\CK
$)
such that 
\[
E_\vartheta(\Gamma_\vth)\overset{\eqref{eq:S_theta}}{=}S_\vth(\Gamma_\vth)
\overset{\eqref{eq:mini-total}}{\le} S_\vth(\beta)=E_\vth(
\beta)
\]
for all $D_2$-symmetric $\beta\in\CK$,
since such $\beta $ are contained in $\Sigma_\mathcal{K}$.
By definition of the total energy $E_\vth$ we infer a uniform bound
on the bending energies $E(\Gamma_\vth)$,
\begin{equation}\label{eq:mini-bending}
E(\Gamma_\vth)\le E_\vth(\Gamma_\vth)\le
E_\vth(\beta)\le
E(\beta)+\TP_q^{\frac1{q-2}}(\beta)<\infty
\,\,\Foa\vth\in (0,1]
\end{equation}
for all $D_2$-symmetric curves $\beta\in\CK$. Notice that the right-hand side is
finite by virtue of
Theorem~\ref{thm:energy-space}.
Together with the uniform $L^\infty$-bound
$ \norm[L^\infty(\R/\Z,\R^3)]{\Gamma_\vth}\le 1/4, $
which follows from Lemma~\ref{lem:improved-Linfty-bound} since $\mathscr{L}
(\Gamma_\vth)=1$,
we obtain the uniform bound
\begin{equation}\label{eq:uniform-norm}
\norm[W^{2,2}(\R/\Z,\R^3)]{\Gamma_\vth}\le C<\infty\quad\Foa \vth\in (0,1].
\end{equation}
Hence, for any given sequence $\vth_j\to 0$ there exists a
subsequence $(\vth_{j_k})_k\subset (\vth_j)_j$ such that
the corresponding symmetric minimizing knots $\Gamma_{\vth_{j_k}}$
converge weakly in $W^{2,2}$ and strongly in $C^1$ to a limiting
curve $\Gamma_0$ as $k\to\infty$. Therefore $\Gamma_0$
satisfies $|\Gamma_0'|\equiv 1$
on $\R/\Z$, $\mathscr{L}(\Gamma_0)=1$, the
dihedral symmetry relation $\tau^{1}_{d_i}(\Gamma_0)=\Gamma_0$ for $i=0,1,2,3$.
By the lower semicontinuity
of the bending energy $E$ and by means of \eqref{eq:mini-bending},
\begin{align*}
E(\Gamma_0) & \le \liminf_{k\to\infty}E(\Gamma_{\vth_{j_k}})
\le \liminf_{k\to\infty}E_{\vth_{j_k}}(\Gamma_{\vth_{j_k}})
\overset{\eqref{eq:mini-bending}}{\le}
\liminf_{k\to\infty}E_{\vth_{j_k}}(\beta)=E(\beta)
\end{align*}
for all $D_2$-symmetric curves  $\beta\in\CK$, which is
the minimizing property 
 \eqref{eq:E_b-symm-minimizer}.
\end{proof}

\begin{proof}[Proof of Theorem~\ref{thm:existence-total}]
Since $\Sigma_{\mathcal{K}}$ was assumed to be non-empty, we have 
$\inf_{\Sigma_\mathcal{K}}S_\vartheta\in [(2\pi)^2,\infty)$, where we 
used  that the Sobolev space 
$W^{2,2}$ continuously embeds\footnote{%
{This follows, e.g., by the characterization of %
Sobolev spaces of real positive smoothness in terms of
Triebel--Lizorkin and Besov spaces~\cite[Prop.~2.1.2]{runst-sickel_1996}
applied to an embedding with constant differential dimension~\cite[Rem.~2.2.3/2]{runst-sickel_1996}.}}
into the fractional Sobolev space
$W^{2-(1/q),q}$ for $q\in (2,4]$
so that the tangent-point energy of a
regular embedded $W^{2,2}$-curve is finite according to
Theorem~\ref{thm:energy-space}.
Hence there exists
a minimal sequence $(\g_j)_j\subset\Sigma_{\mathcal{K}}$ with $
\lim_{j\to\infty}S_\vartheta(\g_j)=\inf_{\Sigma_\mathcal{K}}S_\vartheta.
$
Due to the scale-invariance of $S_\vartheta$ we may assume that $\mathscr{L}(\g_j)=1$
for all $j\in\N$, and we can reparametrize to arclength to obtain a minimal sequence
$\Gamma_j$ with $\abs{\Gamma_j'}=1$ for all $j$, and with
 \begin{equation}\label{eq:minimal-seq}
\lim_{j\to\infty}E_\vartheta(\Gamma_j)=\lim_{j\to\infty}S_\vartheta(\Gamma_j)=
\inf_{\Sigma_\mathcal{K}}S_\vartheta(\cdot).
\end{equation}
Note that the first equation holds since we have $\mathscr{L}(\Gamma_j)=1$ for all 
$j\in\N$.
Moreover, $\Gamma_j\in\Sigma_\mathcal{K}$ for all $j$ due to Lemma~\ref{lem:arclength-sym} for $\ell=L=1$, and therefore, $\|\Gamma_j\|_{L^\infty}\le 1/4$ for 
all $j\in\N$ by virtue of Lemma~\ref{lem:improved-Linfty-bound}.
Now \eqref{eq:minimal-seq} implies that 
\[
\int_{\R/\Z}\abs{\Gamma_j''(s)}^2\,ds=E(\Gamma_j)\le E_\vartheta(\Gamma_j)\le
\inf_{\Sigma_\mathcal{K}}S_\vartheta(\cdot)+1<\infty\Foa j\gg 1,
\] 
which together with the
uniform $L^\infty$-bound and with $\abs{\Gamma_j'}\equiv 1$ for all $j$, yields a uniform bound on the full $W^{2,2}$-norm
of the $\Gamma_j$ for $j\gg 1.$  Consequently, there exists a subsequence $(\Gamma_{j_k})_k
\subset (\Gamma_j)_j$ converging weakly in $W^{2,2}$ and strongly
in $C^1$ to a limit curve $\Gamma_\vartheta\in W^{2,2}(\R/\Z,\R^3)$ as $k\to
\infty$. The $C^1$-convergence implies that $\abs{\Gamma_\vartheta'}\equiv 1$, and 
that $\mathscr{L}(\Gamma_\vartheta)=1$. Moreover, taking the limit $k\to\infty$ in the symmetry relation
\[
\tau^1_{d_i}(\Gamma_{j_k})(t)=\Gamma_{j_k}(t)\quad\Foa t\in\R/\Z
\]
we find  $\tau^1_{d_i}(\Gamma_\vartheta)=\Gamma_\vartheta$. To prove that
$\Gamma_\vartheta$ is contained in  
$\Sigma_\mathcal{K}$ it suffices 
to show that $\Gamma_\vth$ is embedded since
then  $[\Gamma_\vartheta]=
\mathcal{K}$ by
Lemma~\ref{lem:isot}.
The uniform $W^{2,2}$-bound on the $\Gamma_{j_k}$ implies by the
Morrey--Sobolev embedding also a uniform bound on the
$W^{2-(1/q),q}$-norms of the $\Gamma_{j_k}$. This in turn
yields a uniform positive  lower bound $B$ on the the bi-Lipschitz constants 
$\BiLip(\Gamma_{j_k})$ according to Lemma~\ref{lem:bilipschitz}.
{Passing} to the limit $k\to\infty$ in the corresponding inequality
\[
\abs{\Gamma_{j_k}(s)-\Gamma_{j_k}(t)}\ge B\abs{s-t}_{\R/\Z}\quad\Foa
s,t\in\R/\Z
\]
one obtains from the $C^1$-convergence $\Gamma_{j_k}\to\Gamma_\vth$
\[
\abs{\Gamma_\vth(s)-\Gamma_\vth(t)}\ge B\abs{s-t}_{\R/\Z}\quad\Foa
s,t\in\R/\Z.
\]
 By Theorem~\ref{thm:reg-tanpoint} the tangent-point energy is lower-semicontinuous
 with respect to the strong $C^1$-convergence, and therefore
 also the  total scaled energy $S_\vartheta$  with respect to the combined weak $W^{2,2}$- and strong $C^1$-convergence,  which implies
 by virtue of the fact that $\mathscr{L}(\Gamma_\vartheta)=1$,
\begin{equation*}%
\inf_{\Sigma_\mathcal{K}}S_\vth \le S_\vth(\Gamma_\vth)= 
E_\vartheta(\Gamma_\vartheta)
\le\liminf_{k\to\infty}S_\vartheta(\Gamma_{j_k})=
\inf_{\Sigma_\mathcal{K}}S_\vartheta(
\cdot).\qedhere
\end{equation*}
\end{proof}

We can now identify the shape of the $D_2$-elastic unknot  as the once covered circle
contained in the $\mathbf{e_1}$-$\mathbf{e_3}$-plane with starting point on the $\mathbf{e_3}$-axis. This information is even more concrete than stated in
Theorem~\ref{thm:elasym-unknot}, because of our choice  of rotational axes in
\eqref{eq:group-action} representing the dihedral
group $D_2$ on $\R^3$ and of the dihedral parameter transformations  of the domain $\R/\Z$ described in \eqref{eq:inner-action}.

\begin{proof}[Proof of Theorem~\ref{thm:elasym-unknot}]
The once covered circle of length one uniquely minimizes the 
bending energy $E$ in $\CK$ {where $\mathcal{K}$ is the unknot class}
according to the stability
result by Langer and Singer \cite{langer-singer_1985}.
But it also uniquely minimizes the tangent-point energy $\TP_q$  by
the two uniqueness proofs of  Volkmann and  Blatt; see 
\cite[Cor.~5.12]{volkmann_2016}. Therefore 
the once covered circle 
of length one and all its isometric images
also uniquely minimize the total energy $E_\vth$
within $\mathscr{C}(\mathcal{K})$ for any $\vth>0$.
{The dihedral symmetry forces the $E_\vth$-minimizing circles  to lie
in the $\mathbf{e_1}$-$\mathbf{e_3}$-plane, with initial point contained
in $\R\mathbf{e_3}$; see Corollary~\ref{cor:circle}. These properties transfer  
via $C^1$-convergence of $E_{\vth_j}$-minimizers as $\vth_j\to 0$ to
the elastic unknot. }

{Now assume that a $D_{2}$-elastic knot for some 
knot class~$\mathcal K$ is the once covered circle, which represents
the unknot class. According to Lemma~\ref{lem:isot} there is an entire
$C^{1}$-neighborhood of the once covered circle that only
consists of unknots. But by definition of elastic knots this once covered circle
is the $C^1$-limit of $E_{\vth_j}$-minimizing knots $\Gamma_{\vth_j}$,
$\vth_j\to 0$, all representing
the knot  class~$\mathcal{K}$.
 This implies that $\mathcal K$ is the unknot.}

{Notice finally, for the proof of \eqref{eq:infimum-minimal-unknot}, that
Fenchel's lower bound of $2\pi$ for the total curvature of any closed curve 
combined
with H\"older's inequality implies that $(2\pi)^2$ is
 the infimal bending energy for the trivial knot class. This value is attained
by the  once covered circle.
Now~\eqref{eq:infimum-minimal-unknot} immediately
follows from the fact that
according to Corollary~\ref{cor:circle}
there are round $D_{2}$-symmetric
circles.}
\end{proof}

Now we turn to non-trivial knot classes satisfying  assumption
\eqref{eq:infimum-minimal} on the infimal bending energy.
{%
Here we need~\cite[Theorem A.1]{gerlach-etal_2017}
where the
F\'ary--Milnor theorem on the lower bound for 
total curvature of non-trivially knotted curves has been {extended to} the 
$C^1$-closure
of knots. We restate it for the convenience of the reader.

\begin{theorem}[F\'ary--Milnor extension]
\label{thm:fary-milnor-extension}
Let $\mathcal{K}$ be a non-trivial (tame) knot class and suppose $\g$ belongs
to the $C^1$-closure of $\mathscr{C}(\mathcal{K})$. Then $\TC(\g)\ge
4\pi.$
\end{theorem}

This result permits to} prove the following
rigidity result, which is the essential ingredient for the proof of
Theorem~\ref{thm:mainthm} presented  in Section~\ref{sec:example}.

\begin{theorem}[Rigidity \& strong convergence]\label{thm:rigidity}
If a knot class $\mathcal{K}$ satisfies
\eqref{eq:infimum-minimal}
then any $D_2$-elastic knot $\Gamma_0$ for $\mathcal{K}$
is (up to  reparametrization)
{the} tangential pair of
co-planar circles with exactly one point in common
{described in Corollary~\ref{cor:circle}}. In addition, any
subsequence of $D_2$-symmetric
$E_\vth$-minimizers $\Gamma_\vth\in\CK$
converges 
strongly in $W^{2,2}$ to {(an isometric image of)} $\Gamma_0$ as $\vth\to 0$.
\end{theorem}

\begin{proof}
{From Theorem~\ref{thm:elasym-unknot} we infer that $\mathcal K$ is nontrivial.}
Applying the extended F\'ary--Milnor Theorem~\ref{thm:fary-milnor-extension}
to any $D_2$-elastic knot $\Gamma_0$ which according to
Theorem~\ref{thm:symmetric-elastic-knots} lies in the $C^1$-closure of $\CK$,
 we estimate by means of H\"older's
inequality
\begin{equation}\label{eq:const-curvature}
(4\pi)^2\overset{\textnormal{Thm.~\ref{thm:fary-milnor-extension}}}{\le}\left(\int_{\Gamma_0}\kappa_{\Gamma_0}\,ds\right)^2\le
E(\Gamma_0)\overset{\eqref{eq:E_b-symm-minimizer}}{\le}\inf_{\beta\in\CK\atop \beta \,\textnormal{$D_2$-symmetric}}E(\beta)\overset{\eqref{eq:infimum-minimal}}{=}(4\pi)^2.
\end{equation}
Consequently, we have equality everywhere, in particular equality in
H\"older's inequality, which implies a constant integrand $\kappa_{\Gamma_0}=4\pi$ a.e. on $\R/\Z$.

Next we prove that $\Gamma_0$ has at least one double point. Indeed,
otherwise, by
Lemma~\ref{lem:isot} the curve 
$\Gamma_0$ would be
contained in $\CK$ since $\Gamma_0$ is the strong
$C^1$-limit of the $E_{\vth_j}$-minimizers $\Gamma_{\vth_j}\in\CK$
 as $\vth_j\to 0$. Combining the minimizing property
\eqref{eq:E_b-symm-minimizer} of $\Gamma_0$ with our assumption
\eqref{eq:infimum-minimal}, which also  implies that $(4\pi)^2=
\inf_{\CK}E$,
 we find that $\Gamma_0$ is an embedded
minimizer of the bending energy within $\mathscr{C}(\mathcal{K})$,
hence a stable critical point of $E$. Applying the stability result
of Langer and Singer \cite{langer-singer_1985}, we find $\Gamma_0$
to be the once covered circle representing the unknot class, 
which contradicts the {fact} that $\mathcal{ K}$
is nontrivial.  
 So, we have shown
that {$\Gamma_{0}$} is not injective.

According to \cite[Cor.~3.4]{gerlach-etal_2017} the elastic $D_2$-symmetric
knot $\Gamma_0$ belongs, up to isometry and reparametrization, to the 
one-parameter family of tangentially intersecting circles $\tpc[\varphi]$
for $\varphi\in [0,\pi]$ explicitly given in 
\cite[Formula (3.2)]{gerlach-etal_2017}. One easily checks that 
the only possible candidates that may respect the $D_2$-symmetry are
the doubly covered circle $\tpc[0]$ and the tangential pair of
co-planar circles $\tpc[\pi]$ with only one touching point. 

But (any isometric image of) 
$\tpc[0]$ is not only $1$-periodic but also $1/2$-periodic, so that
the symmetry assumption $\tau^{1}_{d_2}(\tpc[0])=\tpc[0]$ would lead to
\begin{align}\label{eq:tpc_0-excluded}
\tpc[0](t) & = \tpc[0](t+1/2)=\tau^{1}_{d_2}(\tpc[0])(t+1/2)
\overset{\eqref{eq:group-action}}{=}
R_2\circ\tpc[0]\big(\psi_2^{1}(t+1/2)\big)\notag\\
& \overset{\eqref{eq:inner-action}}{=} R_2\circ \tpc[0](t)\Foa t\in\R/\Z,
\end{align}
so that $\tpc[0](t)\in\R\mathbf{e_2}$ for all $t\in\R/\Z$, which
is a contradiction.

So, (up to isometry) the only remaining option is $\Gamma_0=\tpc[\pi]$, and that this
curve indeed has the $D_2$-symmetry has been verified in Example
\ref{ex:tpc_pi}.

{It remains to establish strong convergence.}
Now that we have identified the   weak $W^{2,2}$-limit of the
$E_{\vth_j}$-minimizers for any sequence $\vth_j\to 0$,
we can use our assumption
\eqref{eq:infimum-minimal}
to find
for any given $\delta >0$ a $D_2$-symmetric curve $\beta\in\CK$
such that $E(\beta)\le (4\pi)^2 +\delta$.
The minimizing property of the $\Gamma_\vth$ yields
\[
(4\pi)^2\le E(\Gamma_\vth)\le E_\vth(\Gamma_\vth)\le E_\vth(\beta)
\le (4\pi)^2 +\delta +\vth \TP_q(\beta)^{\frac1{q-2}}<\infty,
\]
where we used the classic F\'ary-Milnor theorem for the first inequality.
Taking the limit $\vth\to 0$ gives
\[
(4\pi)^2\le\liminf_{\vth\to 0}E(\Gamma_\vth)\le
\limsup_{\vth\to 0}E(\Gamma_\vth)\le (4\pi)^2+\delta\Foa \delta >0.
\]
Therefore, $\lim_{\vth\to 0}E(\Gamma_\vth)=(4\pi)^2=E(\Gamma_0)$, since 
$\kappa_{\Gamma_0}=4\pi$ a.e.\ on $\R/\Z$. This, together with the 
$C^1$-convergence of the $\Gamma_\vth$ to the same limit (up to
{permutation of the axes} and reparametrization) leads to convergence 
in the $W^{2,2}$-norm.
Combining this with the weak convergence to the now unique weak
limit $\Gamma_0$ (up to
{permutation of the axes} and reparametrization) gives
finally strong convergence in $W^{2,2}$ by the subsequence
principle. 
\end{proof}

\section{Infimal bending energy on torus knots with dihedral symmetry}\label{sec:example}

We now investigate the convergence properties of the $D_2$-symmetric torus knots
$g_3^\epsilon$  introduced in Example~\ref{ex:D2-symmetric-torus-knot} fixing $\ell:=1$ and $\rho:=\epsilon^2$.
\begin{lemma}[$W^{2,2}$-convergence of $D_2$-symmetric torus knots]
\label{lem:convergence}
Let  $\ell=1$, $\rho=\epsilon^2,$ and $b\in\Z\setminus\{1,-1\}$ be odd. Then   the
$D_2$-symmetric
 torus knots $g^\epsilon_3\in \Sigma_{\mathcal{T}(2,b)}$ constructed in Example
\ref{ex:D2-symmetric-torus-knot} for $\ell=1$ and $\rho=\epsilon^2$
converge strongly in $W^{2,2}$ to
the tangential pair of co-planar circles $\tpc[\pi]$ (see
Example~\ref{ex:tpc_pi}) as $\epsilon\to 0$.
\end{lemma}

\begin{corollary}[Minimal bending energy for $D_2$-symmetric torus knots]
\label{cor:infimum-trefoil}
The torus knot classes $\TL(2,b)$ for odd $b\in\Z\setminus\{1,-1\}$ are the only knot classes $\mathcal{K}$ that satisfy
\begin{equation}%
\tag{\ref{eq:infimum-minimal}}
\inf_{\beta\in\CK\atop \beta\,\,\textnormal{$D_2$-symmetric}}E(\beta)=(4\pi)^2
=\inf_{\CK}E(\cdot).
\end{equation}
\end{corollary}

\begin{proof}
The convergence of torus knots established in Lemma~\ref{lem:convergence} together with an appropriate rescaling to unit length, i.e., the $W^{2,2}$-convergence of
$g_3^\epsilon /\mathscr{L}(g_3^\epsilon)\in\Sigma_{\mathcal{T}(2,b)}$ to $\tpc[\pi]$,
allows us to identify the infimal bending energy $(4\pi)^2=E(\tpc[\pi])$
on the class of $D_2$-symmetric curves in $\mathscr{C}(\TL(2,b))$, since by
the Morrey--Sobolev embedding we have $g_3^\epsilon\to\tpc[\pi]$ in $C^1$, and hence $\mathscr{L}(g^\eps_3)\to\mathscr{L}(\tpc[\pi])=1$, and therefore also
$g_3^\epsilon/\mathscr{L}(g_3^\epsilon)\to\tpc[\pi]$ in $W^{2,2}$ as $\epsilon\to 0$.
That the torus knot classes $\TL(2,b)$ are the only possible knot classes
to satisfy the right equality in~\eqref{eq:infimum-minimal} was proven in 
\cite[Corollary 4.4]{gerlach-etal_2017}.
\end{proof}

\begin{proof}[Proof of Lemma~\ref{lem:convergence}]
For simplicity we  restrict the explicit arguments to the case that
{$b\ge3$}.
By symmetry it suffices to prove the $W^{2,2}$-convergence  of the 
generating arcs $\alpha^\epsilon:=\alpha_3^\epsilon$ of $g^\epsilon_3$
to the corresponding generating arc $\alpha_2$ of $\tpc[\pi]$ defined in
Example~\ref{ex:tpc_pi}. 
Since $\alpha_3^\epsilon$ is piecewise defined (see \eqref{eq:generating-torus-knot}) we focus first on the interval $I_1(\epsilon):=[0,(b+1)\epsilon/2]$ where $\alpha_3^\epsilon=h^\epsilon$, and obtain by direct computation from the explicit 
expressions \eqref{eq:helical-part} and \eqref{eq:phi} for the helical part $h^\epsilon$ and 
the parameter transformation $\phi$ (for $\rho=\epsilon^2$)
\begin{equation}\label{eq:second-deriv}
\abs{(h^\epsilon)''(t)}^2=\pi^4\phi'^4(t/\epsilon)+\pi^2
\phi''^2(t/\epsilon)\quad\Foa t\in  I_1(\epsilon),
\end{equation}
which implies
\begin{align}
\label{eq:L2-curvature0}
\norm[L^2(I_1(\epsilon),\R^3)]{(h^\epsilon)''}^2 &
\le\epsilon\pi^2\int_0^{(b+1)/2}\big(\pi^2\phi'^4(z)+1\big)\,dz,
\end{align}
where we changed variables to $z:=t/\epsilon$ and used that $ \abs{\phi''}
\le1$
on $[0,(b+1)/2)$.
Now, $\phi'=1$ on $ [0,(b-1)/2]$ whereas $\phi'(t)=-(t-(b+1)/2)$ for $
t\in [(b-1)/2,(b+1)/2]$ according to \eqref{eq:phi} so that we obtain from
\eqref{eq:L2-curvature0}
\begin{equation}\label{eq:L2-helix}
\norm[L^2(I_1(\epsilon),\R^3)]{(h^\epsilon)''}^2 \le\epsilon\pi^2\left[
\pi^2\left(\frac{b-1}2+\frac15\right)+\frac{b+1}2\right]<\epsilon\pi^4(b+1).
\end{equation}
The prospective arclength parametrized limit curve $\tpc[\pi]$ has constant curvature $\abs{\tpc[\pi]''}=4\pi$
a.e., so that by virtue of \eqref{eq:L2-helix}
\begin{equation}\label{eq:first-deviation}
\norm[L^2(I_1(\epsilon),\R^3)]{(\alpha^\epsilon)''-\alpha_2''}^2
\overset{\eqref{eq:L2-helix}}{\le} 2\epsilon\pi^4(b+1)+(4\pi)^2(b+1)\epsilon.
\end{equation}
Now we consider the interval $I_2(\epsilon):=[(b+1)\epsilon/2,(1/4)-(b+1)\epsilon/2]$ and recall
our condition \eqref{eq:r-condition} on the radius $r$ of the stadium curve $\sigma^\epsilon$, namely (now for $\ell=1$) 
\begin{equation}\label{eq:r-condition-1}
(b+1)\epsilon +\pi r =1/4.
\end{equation}
{This} implies that the auxiliary function $f(t,s):=\abs{r^{-1}
(t-(b+1)s/2)-4\pi t  }^2$ satisfies for any $\epsilon\in (0,(8(b+1))^{-1})$
\begin{equation}\label{eq:auxiliary-est}
0\le f(t,s)< {1024}\pi^2(b+1)^2\epsilon^2\quad\textnormal{
for all $(t,s)\in [0,1]\times [0,2\epsilon]$}
\end{equation}
since $\sqrt{f(t,s)}<4\pi(b+1)\abs{8\eps t-s}
\le 4\pi(b+1)\max\br{8\eps,2\eps}
\le 32\pi(b+1)\eps$.
Inequality~\eqref{eq:auxiliary-est}
can be used to estimate
\begin{align*}
\|(\alpha^\epsilon)''-\alpha_2''\|_{L^2(I_2(\epsilon),\R^3)}^2
& \le 2\int_{I_2(\epsilon)}
\big|(1/r) -4\pi\big|^2\,dt\\
 +2\cdot(4\pi)^2 \int_{I_2(\epsilon)}
&\left|\begin{pmatrix}
\cos((t-(b+1)\epsilon/2)/r)-\cos(4\pi t)\\
-\sin((t-(b+1)\epsilon/2)/r)+\sin(4\pi t)\\
0\end{pmatrix}\right|^2\,dt.\notag 
\end{align*}
The first integrand equals $f(1,0)$, and the second
integrand can be estimated from above by $2f(t,\epsilon)$ for $t\in I_2(\epsilon)$, since both, $\cos $ and $\sin$, have Lipschitz 
constant $1$. Therefore, we can apply the auxiliary estimate 
\eqref{eq:auxiliary-est} to arrive at
\begin{equation}\label{eq:second-deviation}
\|(\alpha^\epsilon)''-\alpha_2''\|_{L^2(I_2(\epsilon),\R^3)}^2
< \frac12(1+2(4\pi)^2) {1024}\pi^2(b+1)^2\epsilon^2,
\end{equation}
where we also used that $\mathscr{L}^1(I_2(\epsilon))< 1/4.$
Finally, on the interval $I_3(\epsilon):=[(1/4)-(b+1)\epsilon/2,1/4]$ the
stadium curve $\sigma|_{I_3(\epsilon)}=\alpha^\epsilon$ is a straight segment so that
\begin{equation}\label{eq:third-deviation}
\|(\alpha^\epsilon)''-\alpha_2''\|_{L^2(I_3(\epsilon),\R^3)}^2=
\|\alpha_2''\|_{L^2(I_3(\epsilon),\R^3)}^2=(4\pi)^2(b+1)\epsilon/2.
\end{equation}
Summarizing \eqref{eq:first-deviation}, \eqref{eq:second-deviation}, and
\eqref{eq:third-deviation} we obtain   a constant $C_1\ge 1$ independent of $\epsilon$
such that 
\begin{equation}\label{eq:L2-convergece}
\|(\alpha^\epsilon)''-\alpha_2''\|_{L^2([0,1/4],\R^3)}\le C_1\sqrt{\epsilon}\quad\Foa 
0< \epsilon <\frac1{8(b+1)},
\end{equation}
which by means of Poincar\'e's inequality~\cite[Section 5.8.1]{evans_1998}
applied to $\g'$ (satisfying
$\int_{\R/\Z}\g'(\tau)\,d\tau=0$ because $\g$ is $1$-periodic) implies that
there is a constant $C_2\ge 1$ such that
\begin{equation}\label{eq:W12-est}
\|(\alpha^\epsilon)'-\alpha_2'\|_{W^{1,2}([0,1/4],\R^3)}\le C_2\sqrt{\epsilon}
\quad\Foa 0< \epsilon <\frac1{8(b+1)}.
\end{equation}
To conclude the proof it therefore suffices to prove the uniform convergence of the
$\alpha^\epsilon$ to $\alpha_2$ on $[0,1/4]$ as $\epsilon\to 0$. 
We have for any $t\in[0,\frac14]$
\begin{align*}
 \abs{\alpha^\epsilon(t)-\alpha_2(t)}
 &=\abs{\alpha^\epsilon(0)-\alpha_2(0)} + \abs{\int_{0}^{t}\br{\br{\alpha^\epsilon}'-\alpha_2'}} \\
 &=\rho + \sqrt t\norm[L^{2}]{\br{\alpha^\epsilon}'-\alpha_2'}
 \le \epsilon^{2} + \tfrac12 C_{2}\sqrt\epsilon.
\end{align*}
Taking the supremum over $t\in[0,\frac14]$ concludes the proof.
\end{proof}

\begin{proof}[Proof of Theorem~\ref{thm:mainthm}]
(i)
Corollary~\ref{cor:infimum-trefoil} implies that the torus knot classes $\TL(2,b)$ for odd 
$b\in\Z\setminus\{1,-1\}$ satisfy condition \eqref{eq:infimum-minimal}, so that 
Theorem~\ref{thm:rigidity} is applicable.

(ii) If a $D_{2}$-symmetric elastic knot
for some knot class $\mathcal{K}$
is $\tpc$
then there is a sequence of curves 
$\seqn\gamma\subset \mathscr{C}(\mathcal{K})$
such that $\gamma_{k}\to\tpc[\pi]$ with respect to the
$C^{1}$-norm.
{For all $\gamma_{k}$  sufficiently
close to $\tpc[\pi]$ with respect to the $C^{1}$-norm,
we obtain some cumulative angle $\Delta\beta$
as described in~\cite[Prop.~4.2]{gerlach-etal_2017}
which yields the existence of some odd integer $b$ such that
$\gamma_{k}$ is a $(2,b)$-torus knot if $\abs b\ge3$
and unknotted if $b=\pm1$.
The latter is ruled out by Theorem~\ref{thm:elasym-unknot}.

(iii) This is an immediate consequence {of} Corollary~\ref{cor:infimum-trefoil}.}
\end{proof}

\section{{Discussion and open problems}}\label{sec:discussion}

\subsection{{Higher regularity of $D_2$-symmetric $E_\vth$-critical knots}}\label{sec:smoothness}
 {The Euler--Lagrange operator of $\TP_q$ for $q>2$ studied in \cite{blatt-reiter_2015a}
 seems to be} related
 to the $q$-Laplacian for which
 one cannot expect full regularity, at least in the non-fractional case.
 So it is open whether {the} arclength parametrized $D_2$-symmetric
 critical points of the constrained variational problem~\eqref{eq:P_theta} {obtained
 in Theorem~\ref{thm:symmetric-critical}}   are of class $C^\infty(\R/\Z,\R^3)$. 
 
{Choosing instead of $\TP_q$} the decoupled tangent-point functionals $\TP^{(p,2)}$ {for  $p\in (4,5)$}, cf.\ Footnote~\ref{foot:decoup},
 whose domain is a Hilbert space,
 we can generalize the bootstrapping argument from~\cite{blatt-reiter_2015a} to obtain  $C^{\infty}$-regularity.
 We might even derive analyticity by extending the arguments given
 in~\cite{MR3884791,vorderobermeier,steenebruegge-vorderobermeier}.

\subsection{Non-embeddedness of symmetric elastic knots}\label{sec:double-points}
{Similarly as in \cite[Proposition 3.1]{gerlach-etal_2017} we expect that also symmetric elastic knots for non-trivial knot classes must have
double points.}
According to the stability result of Langer and Singer \cite{langer-singer_1985}, the only stable critical point {of the bending energy $E$} is the once covered circle.
{However, due} to the fact that the symmetry constraint only permits to apply symmetry preserving variations in \eqref{eq:E_b-symm-minimizer}, we cannot immediately apply this tool in the present work.
Consequently, in contrast to {the case of (not necessarily symmetric) elastic knots treated} in~\cite{gerlach-etal_2017}, we {are presently not able to show that}
\begin{itemize}
\item every  $D_2$-elastic knot for a non-trivial
knot class $\mathcal{K} $ must have self-intersection points, {and that}
\item inequality \eqref{eq:E_b-symm-minimizer} is strict unless $\mathcal{K}$ is the unknot class.
\end{itemize}
It is an interesting question whether one may derive a weaker
version {of the stability result} in~\cite{langer-singer_1985}, that is applicable
in our situation,
e.g., stating that the round circle would be the only local minimizer within 
{the $D_2$-symmetric subclass $\Sigma^{1}$ introduced in Definition \ref{def:dihedral}.}
In this case,
one could argue as in~\cite{gerlach-etal_2017}
to conclude that any embedded minimizer of the bending energy
within {the set $\Sigma_{\mathcal K}$ of $D_2$-symmetric knots defined in \eqref{eq:symmetric-set}}
would in fact be a local minimizer within $\Sigma^{1}$. Hence, $\mathcal K$ would be the unknot {contradicting the assumption of a non-trivial knot class $\mathcal{K}$.}  Thus, the {$D_2$-symmetric minimizer of the bending energy could not
be embedded, and the} infimum {of the bending energy could not be attained in $\Sigma_\mathcal{K}$.}

\subsection{Other knot classes and symmetries}\label{sec:otherknots}

In a similar manner as in Definition~\ref{def:symmetric-elastic-knots} we may define \emph{$G$-elastic
knots} for any symmetry group~$G$.
Most results 
{from the Introduction, namely Theorems~\ref{thm:symmetric-critical},
\ref{thm:symmetric-elastic-knots},
and~\ref{thm:elasym-unknot}
as well as Corollary~\ref{cor:symmetric-critical}}, carry over
to  more general symmetry groups
{while Theorem~\ref{thm:mainthm} is restricted to~$D_{2}$}.
We briefly speculate about some 
{examples}
involving other knot classes or symmetry
groups different from $D_{2}$.

\begin{figure}\tiny
\begin{tabular}{c@{\qquad}c@{\qquad}c}
\includegraphics[scale=.5,trim=90 110 90 70,clip]{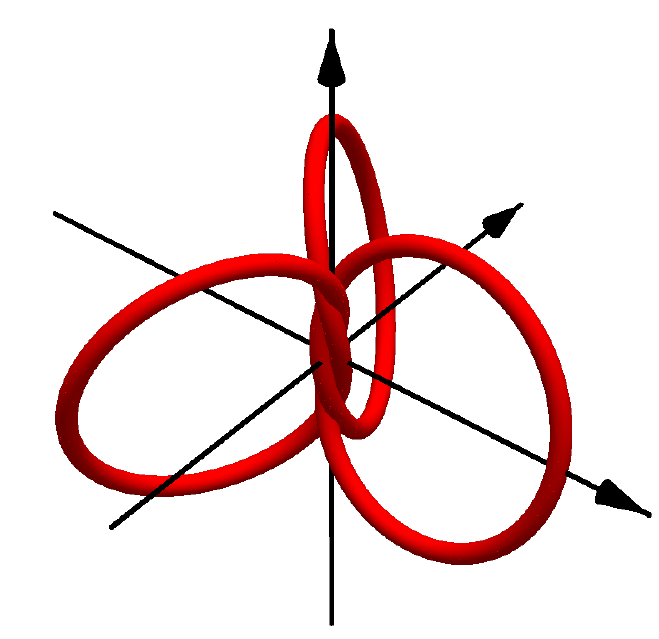}&
\includegraphics[scale=.2,trim=0 20 0 0,clip]{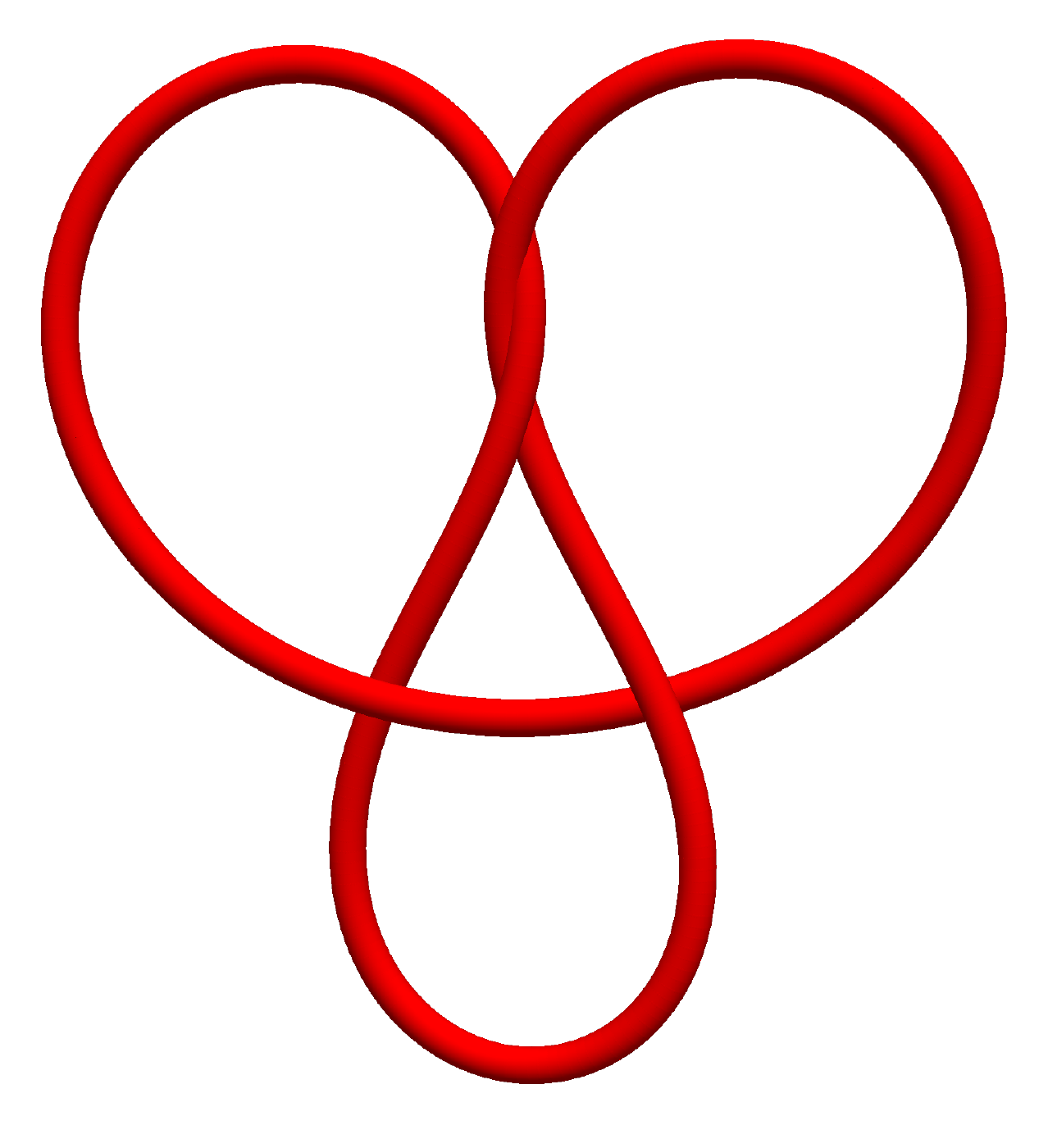}&
\includegraphics[scale=.3,trim=30 50 30 50,clip]{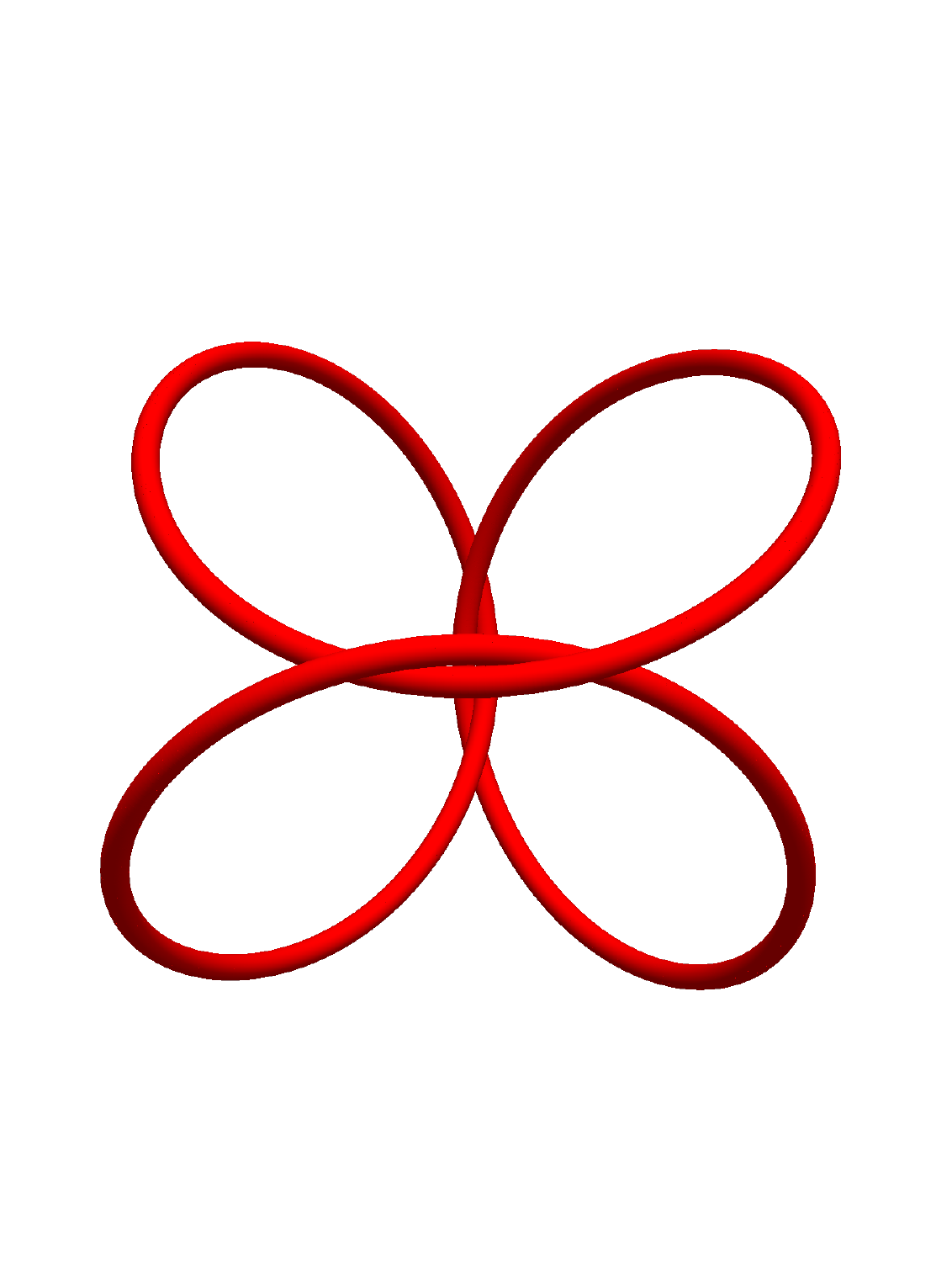} \\
(A)&(B)&(C)
\end{tabular}
\caption{Numerical approximations of candidates for (A)~the $D_{3}$-elastic $(3,4)$-torus knot, (B)~the elastic figure-eight knot,
(C)~the $D_{2}$-elastic figure-eight knot.}
\label{fig:threefoil}
\end{figure}

\subsection*{General torus knots}
Let $a,b\in\N$ be coprime with $2\le a<b$.
We expect the elastic knot to be the $a$-times covered circle~\cite{gerlach-etal_2017}
which then would agree with the $D_{b}$-elastic knot.

Consequently, the $D_{a}$-elastic knot is likely to be the union of
$a$ circles of radius $1/(2\pi a)$ that tangentially meet
in one common point. The angle between two consecutive circles
amounts to $2\pi/a$, {which is also observed experimentally by means of 
 the numerical gradient flow\footnote{An animation is attached as an ancillary file to the present arXiv submission.} of Bartels et al.; see}
Figure~\ref{fig:threefoil}~(A).

\subsection*{The figure-eight knot}
Simulations that have been carried out earlier
suggest that an elastic figure eight ($4_{1}$)
is either planar~\cite{avvakumov-sossinsky_2014}
or spherical~\cite{gallotti-pierre-louis_2007,gerlach-etal_2017};
see Figure~\ref{fig:threefoil}~(B) and~(C).
Recent numerical experiments~\cite{bartels-reiter_2018}
support the {former}.
So the planar configuration~(B) could be a global minimizer
within the figure-eight class
whereas the spherical configuration~(C)
might merely be a local minimizer.

Assuming that this is true and that there are no further candidates,
we may conjecture that the planar configuration is an elastic knot, while the spherical configuration is a $D_{2}$-elastic knot. 
In contrast to the latter, the former does not enjoy a
$D_2$-symmetry.
Using the theory developed above, we can only
state that there exists a $D_{2}$-elastic figure-eight knot
which may or may not coincide with the {elastic} figure-eight.

\addtocontents{toc}{\SkipTocEntry}
 \section*{Acknowledgments}
We are indebted to S\"oren Bartels for
 fruitful discussions on the numerical approximation
 of elastic knots.
 {While working on this paper, the first author was an International Research Fellow of JSPS (Postdoctoral Fellowships for Research in Japan).}
The second author has been partially supported by Grant RE~3930/1--1 of the German Research Foundation (DFG).
The third author's work is partially funded by DFG Grant no.\@ Mo 966/7-1 \emph{Geometric curvature functionals: energy landscape and discrete methods} (project
no. 282535003),
and by the Excellence Initiative of the German federal and state governments.
{A meeting of all three authors in 2020 at RWTH Aachen University was generously funded by {the DFG}-Graduiertenkolleg \emph{Energy, Entropy, and Dissipative Dynamics (EDDy)}},  project no. 320021702/GRK2326.

\addtocontents{toc}{\SkipTocEntry}


\begin{thebibliography}{10}

\bibitem{adams_2004}
C.~C. Adams.
\newblock {\em The knot book}.
\newblock American Mathematical Society, Providence, RI, 2004.

\bibitem{alt_2016}
H.~W. Alt.
\newblock \href {http://dx.doi.org/10.1007/978-1-4471-7280-2} {{\em Linear
  functional analysis}}.
\newblock Springer-Verlag London, 2016.
\newblock Translated by R.~N\"{u}rnberg.

\bibitem{avvakumov-sossinsky_2014}
S.~Avvakumov and A.~Sossinsky.
\newblock \href {http://dx.doi.org/10.1134/S1061920814040013} {On the normal
  form of knots}.
\newblock {\em Russ. J. Math. Phys.}, 21(4):421--429, 2014.

\bibitem{bartels_2013}
S.~Bartels.
\newblock \href {https://doi.org/10.1093/imanum/drs041} {A simple scheme for
  the approximation of the elastic flow of inextensible curves}.
\newblock {\em IMA J. Numer. Anal.}, 33(4):1115--1125, 2013.

\bibitem{knotevolve}
S.~Bartels, {\relax Ph}.~Falk, and P.~Weyer.
\newblock \href {https://aam.uni-freiburg.de/agba/forschung/knotevolve/}
  {{KNOT}evolve -- a tool for relaxing knots and inextensible curves}.
\newblock Web application, 2020.

\bibitem{bartels-reiter_2018}
S.~{Bartels} and {\relax Ph}.~{Reiter}.
\newblock \href {http://dx.doi.org/10.1090/mcom/3633} {{Stability of a simple
  scheme for the approximation of elastic knots and self-avoiding inextensible
  curves}}.
\newblock {\em arXiv e-prints}, 2018.
\newblock To appear in \emph{Mathematics of Computation}.

\bibitem{bartels-etal_2018}
S.~Bartels, {\relax Ph}.~Reiter, and J.~Riege.
\newblock \href {http://dx.doi.org/10.1093/imanum/drx021} {A simple scheme for
  the approximation of self-avoiding inextensible curves}.
\newblock {\em IMA J. Numer. Anal.}, 38(2):543--565, 2018.

\bibitem{blatt_2009}
S.~Blatt.
\newblock \href
  {http://www.instmath.rwth-aachen.de/Preprints/blatt20090825.pdf} {Note on
  continuously differentiable isotopies}.
\newblock Report~34, Institute for Mathematics, RWTH Aachen, 2009.

\bibitem{blatt_2013b}
S.~Blatt.
\newblock The energy spaces of the tangent-point energies.
\newblock {\em J. Topol. Anal.}, 5(261):261--270, 2013.

\bibitem{blatt-reiter_2015a}
S.~Blatt and {\relax Ph}.~Reiter.
\newblock \href {http://dx.doi.org/10.1515/acv-2013-0020} {Regularity theory
  for tangent-point energies: the non-degenerate sub-critical case}.
\newblock {\em Adv. Calc. Var.}, 8(2):93--116, 2015.

\bibitem{blatt-reiter_2015b}
S.~Blatt and {\relax Ph}.~Reiter.
\newblock \href {http://dx.doi.org/10.5186/aasfm.2015.4006} {Towards a
  regularity theory for integral {M}enger curvature}.
\newblock {\em Ann. Acad. Sci. Fenn. Math.}, 40(1):149--181, 2015.

\bibitem{MR3884791}
S.~Blatt and N.~Vorderobermeier.
\newblock \href {http://dx.doi.org/10.1007/s00526-018-1443-6} {On the
  analyticity of critical points of the {M}\"{o}bius energy}.
\newblock {\em Calc. Var. Partial Differential Equations}, 58(1):Paper No. 16,
  28, 2019.

\bibitem{buck-orloff}
G.~Buck and J.~Orloff.
\newblock \href {https://doi.org/10.1016/0166-8641(94)00024-W} {A simple energy
  function for knots}.
\newblock {\em Topology Appl.}, 61(3):205--214, 1995.

\bibitem{burde-zieschang_2003}
G.~Burde and H.~Zieschang.
\newblock {\em Knots}, volume~5 of {\em de Gruyter Studies in Mathematics}.
\newblock Walter de Gruyter \& Co., Berlin, second edition, 2003.

\bibitem{cantarella-etal_2014a}
J.~Cantarella, J.~H.~G. Fu, R.~B. Kusner, and J.~M. Sullivan.
\newblock \href {http://dx.doi.org/10.2140/gt.2014.18.1973} {Ropelength
  criticality}.
\newblock {\em Geom. Topol.}, 18(4):1973--2043, 2014.

\bibitem{cohn_1957}
P.~M. Cohn.
\newblock {\em Lie groups}.
\newblock Cambridge Tracts in Mathematics and Mathematical Physics, no. 46.
  Cambridge University Press, New York, N.Y., 1957.

\bibitem{crowell-fox_1977}
R.~H. Crowell and R.~H. Fox.
\newblock \href {http://link.springer.com/book/10.1007/978-1-4612-9935-6} {{\em
  Introduction to knot theory}}.
\newblock Springer-Verlag, New York-Heidelberg, 1977.
\newblock Reprint of the 1963 original, Graduate Texts in Mathematics, No. 57.

\bibitem{denne-sullivan_2008}
E.~Denne and J.~M. Sullivan.
\newblock \href {http://dx.doi.org/10.1007/978-3-7643-8621-4_8} {Convergence
  and isotopy type for graphs of finite total curvature}.
\newblock In {\em Discrete differential geometry}, volume~38 of {\em
  Oberwolfach Semin.}, pages 163--174. Birkh\"{a}user, Basel, 2008.

\bibitem{DEJvR}
Y.~Diao, C.~Ernst, and E.~J. Janse~van Rensburg.
\newblock \href {http://dx.doi.org/10.1017/S0305004198003338} {Thicknesses of
  knots}.
\newblock {\em Math. Proc. Cambridge Philos. Soc.}, 126(2):293--310, 1999.

\bibitem{evans_1998}
L.~C. Evans.
\newblock {\em Partial differential equations}, volume~19 of {\em Graduate
  Studies in Mathematics}.
\newblock American Mathematical Society, Providence, RI, 1998.

\bibitem{fary_1949}
I.~F\'{a}ry.
\newblock \href {http://www.numdam.org/item?id=BSMF_1949__77__128_0} {Sur la
  courbure totale d'une courbe gauche faisant un n\oe ud}.
\newblock {\em Bull. Soc. Math. France}, 77:128--138, 1949.

\bibitem{fenchel_1929}
W.~Fenchel.
\newblock \href {http://dx.doi.org/10.1007/BF01454836} {\"{U}ber {K}r\"{u}mmung
  und {W}indung geschlossener {R}aumkurven}.
\newblock {\em Math. Ann.}, 101(1):238--252, 1929.

\bibitem{gallotti-pierre-louis_2007}
R.~Gallotti and O.~Pierre-Louis.
\newblock \href {http://dx.doi.org/10.1103/PhysRevE.75.031801} {Stiff knots}.
\newblock {\em Phys. Rev. E (3)}, 75(3):031801, 14, 2007.

\bibitem{gerlach-etal_2017}
H.~Gerlach, {\relax Ph}.~Reiter, and H.~von~der Mosel.
\newblock \href {http://dx.doi.org/10.1007/s00205-017-1100-9} {The elastic
  trefoil is the doubly covered circle}.
\newblock {\em Arch. Ration. Mech. Anal.}, 225(1):89--139, 2017.

\bibitem{gilsbach_2018}
A.~Gilsbach.
\newblock \href
  {http://publications.rwth-aachen.de/record/726186/files/726186.pdf} {{\em On
  symmetric critical points of knot energies}}.
\newblock PhD thesis, RWTH Aachen University, 2018.

\bibitem{gilsbach-vdm_2018}
A.~Gilsbach and H.~von~der Mosel.
\newblock \href {http://dx.doi.org/10.1016/j.topol.2018.04.014} {Symmetric
  critical knots for {O}'{H}ara's energies}.
\newblock {\em Topology Appl.}, 242:73--102, 2018.
\newblock Update on \href{https://arxiv.org/abs/1709.06949}{ArXiv e-prints,
  1709.06949}.

\bibitem{gonzalez-maddocks_1999}
O.~Gonzalez and J.~H. Maddocks.
\newblock \href {http://dx.doi.org/10.1073/pnas.96.9.4769} {Global curvature,
  thickness, and the ideal shapes of knots}.
\newblock {\em Proc. Natl. Acad. Sci. USA}, 96(9):4769--4773 (electronic),
  1999.

\bibitem{gruenbaum-shephard_1985}
B.~Gr{\"u}nbaum and G.~Shephard.
\newblock Symmetry groups of knots.
\newblock {\em Mathematics Magazine}, 58(3):161--165, 1985.

\bibitem{langer-singer_1985}
J.~Langer and D.~A. Singer.
\newblock \href {http://dx.doi.org/10.1016/0040-9383(85)90046-1} {Curve
  straightening and a minimax argument for closed elastic curves}.
\newblock {\em Topology}, 24(1):75--88, 1985.

\bibitem{milnor_1950}
J.~W. Milnor.
\newblock \href {http://dx.doi.org/10.2307/1969467} {On the total curvature of
  knots}.
\newblock {\em Ann. of Math. (2)}, 52:248--257, 1950.

\bibitem{natanson_2016}
I.~P. Natanson and L.~F. Boron.
\newblock {\em Theory of functions of a real variable}.
\newblock Courier Dover Publications, 2016.

\bibitem{nitsche_1971}
J.~C.~C. Nitsche.
\newblock \href {http://dx.doi.org/10.2307/2316484} {The smallest sphere
  containing a rectifiable curve}.
\newblock {\em Amer. Math. Monthly}, 78:881--882, 1971.

\bibitem{palais_1979}
R.~S. Palais.
\newblock \href {http://projecteuclid.org/euclid.cmp/1103905401} {The principle
  of symmetric criticality}.
\newblock {\em Comm. Math. Phys.}, 69(1):19--30, 1979.

\bibitem{reiter_2005}
{\relax Ph}.~Reiter.
\newblock \href
  {http://www.instmath.rwth-aachen.de/Preprints/reiter20051017.pdf} {All curves
  in a {$C^1$}-neighbourhood of a given embedded curve are isotopic}.
\newblock Report~4, Institute for Mathematics, RWTH Aachen, 2005.

\bibitem{runst-sickel_1996}
T.~Runst and W.~Sickel.
\newblock \href {http://dx.doi.org/10.1515/9783110812411} {{\em Sobolev spaces
  of fractional order, {N}emytskij operators, and nonlinear partial
  differential equations}}, volume~3 of {\em De Gruyter Series in Nonlinear
  Analysis and Applications}.
\newblock Walter de Gruyter \& Co., Berlin, 1996.

\bibitem{steenebruegge-vorderobermeier}
D.~Steenebrügge and N.~Vorderobermeier.
\newblock \href {http://arxiv.org/abs/2103.07383} {On the analyticity of
  critical points of the generalized {I}ntegral {M}enger {C}urvature in the
  {H}ilbert case}.
\newblock {\em arXiv e-prints}, 2021.

\bibitem{strzelecki-etal_2013a}
P.~Strzelecki, M.~Szuma{\'n}ska, and H.~von~der Mosel.
\newblock \href {http://dx.doi.org/10.1016/j.topol.2013.05.022} {On some knot
  energies involving {M}enger curvature}.
\newblock {\em Topology Appl.}, 160(13):1507--1529, 2013.

\bibitem{strzelecki-vdm_2012}
P.~Strzelecki and H.~von~der Mosel.
\newblock \href {http://dx.doi.org/10.1142/S0218216511009960} {Tangent-point
  self-avoidance energies for curves}.
\newblock {\em J. Knot Theory Ramifications}, 21(5):1250044, 28, 2012.

\bibitem{volkmann_2016}
A.~Volkmann.
\newblock \href {http://dx.doi.org/10.4310/CAG.2016.v24.n1.a7} {A monotonicity
  formula for free boundary surfaces with respect to the unit ball}.
\newblock {\em Comm. Anal. Geom.}, 24(1):195--221, 2016.

\bibitem{vdm_1998}
H.~von~der Mosel.
\newblock Minimizing the elastic energy of knots.
\newblock {\em Asymptot. Anal.}, 18(1-2):49--65, 1998.

\bibitem{vorderobermeier}
N.~Vorderobermeier.
\newblock \href {http://dx.doi.org/10.1142/S0219199720500455} {On the
  regularity of critical points for {O’Hara’s} knot energies: {F}rom
  smoothness to analyticity}.
\newblock {\em Communications in Contemporary Mathematics}, 2020.

\bibitem{wings_2018}
A.~Wings.
\newblock {S}tetige {D}ifferenzierbarkeit tangentenpunktartiger
  {K}notenenergien.
\newblock Bachelor thesis, RWTH Aachen University, 2014.

\end{thebibliography}
\end{document}